\documentclass[final, 3p, times]{elsarticle}%
\usepackage{graphicx}
\usepackage{xcolor}
\usepackage{colortbl}
\usepackage{epstopdf}
\usepackage{amssymb}
\usepackage{amsthm}
\usepackage{amsmath}
\usepackage{color}
\usepackage{dsfont}
\usepackage{mathtools}
\usepackage{extarrows}
\usepackage{hyperref}
\usepackage{bm}
\usepackage{caption}
\usepackage{float}
\usepackage{verbatim}
\usepackage{epsfig}
\usepackage{multirow}
\usepackage{subfig}
\usepackage[section]{placeins}
\usepackage{cases}

\newtheorem{lemma}{Lemma}[section]
\newtheorem{theorem}{Theorem}[section]
\newtheorem{corollary}{Corollary}[section]

\newtheorem{example}{Example}[section]
\biboptions{compress}

\journal{XXX}

\begin{document}

\begin{frontmatter}
\title{Mittag-Leffler stability of complete monotonicity-preserving schemes for time-dependent coefficients sub-diffusion equations}

\tnotetext[label1]{The research of Dongling Wang is supported in part by NSFC (No. 12271463) and Outstanding Youth Foundation of Department of Education in Hunan Province (No. 22B0173).\\ Declarations of interest: none.}

\author[NWU]{Wen Dong} 
\ead{dongwen0724@163.com}
\author[XTU]{Dongling Wang\corref{mycorrespondingauthor}}
\ead{wdymath@xtu.edu.cn, ORCID:0000-0001-8509-2837}
\cortext[mycorrespondingauthor]{Corresponding author. }

\address[NWU]{School of Mathematics, Northwest University, Xi'an, Shaanxi, 710127, P.R. China}
\address[XTU]{School of Mathematics and Computational Science, Xiangtan University, Xiangtan, Hunan 411105, China}

\begin{abstract}
A key characteristic of the anomalous sub-solution equation is that the solution exhibits algebraic decay rate over long time intervals, which is often refered to the Mittag-Leffler type stability. For a class of power nonlinear sub-diffusion models with variable coefficients, we prove that their solutions have Mittag-Leffler stability when the source functions satisfy natural decay assumptions. That is the solutions have the decay rate $\|u(t)\|_{L^{s}(\Omega)}=O\left( t^{-(\alpha+\beta)/\gamma} \right)$ as $t\rightarrow\infty$, where $\alpha$, $\gamma$ are positive constants, $\beta\in(-\alpha,\infty)$ and $s\in (1,\infty)$. Then we develop the structure preserving algorithm for this type of model. For the complete monotonicity-preserving ($\mathcal{CM}$-preserving) schemes developed by Li and Wang (Commun. Math. Sci., 19(5):1301-1336, 2021), we prove that they satisfy the discrete comparison principle for time fractional differential equations with variable coefficients. Then, by carefully constructing the fine the discrete supersolution and subsolution, we obtain the long time optimal decay rate of the numerical solution $\|u_{n}\|_{L^{s}(\Omega)}=O\left( t_n^{-(\alpha+\beta)/\gamma} \right)$ as $t_{n}\rightarrow\infty$, which is fully agree with the theoretical solution. 
Finally, we validated the analysis results through numerical experiments.

\end{abstract}


\end{frontmatter}

\section{Introduction}\label{Introd}

A typical application of time fractional equations is to describe various anomalous diffusion models, namely sub-diffusion and super-diffusion phenomena.  For the standard diffusion process, which is usually described by the classical heat equation,  where the mean square displacement (MSD) of particles is a linear function of time, i.e., $\langle x^2(t)\rangle \sim Ct$.
\footnote{As usual, our notation $\sim$ here represents the asymptotic limit $\lim_{t\to \infty} \frac{\langle x^2(t)\rangle}{t}=C$, where $C$ represents a general positive constant, which may take different values in different places, but is always independent of $t$ or $n$. We also use the standard notation $O$ to represent a certain asymptotic order, for example $y(t)=O(t^{-\alpha})$ means $\lim_{t\to \infty} \left| \frac{y(t)}{t^{-\alpha}}\right| \leq C$.}
However, many actual model data indicate that the MSD of anomalous diffusion processes may have a form $\langle x^2(t)\rangle \sim Ct^{\alpha}$, where $\alpha\in(0, 2)$. Based on random walk theory, we can rigorously derive fractional diffusion equations, where Caputo derivatives are typically used in the time direction instead of classical first-order derivatives \cite{sub-diff2, caputo}.
Specifically, when $\alpha\in(0, 1)$, it is referred to as the sub-diffusion process.

For the time fractional order model, its solutions exhibit asymptotic behavior completely different from the standard diffusion equation over a long time interval.  
In fact, solutions to time fractional equations typically have algebraic decay rates, while standard first-order equations have exponential decay rates. This can be seen from the simple example below. Consider the fractional ODE: $D_{t}^{\alpha}y(t)=\lambda y(t)$ with $y(0)=y_0>0$,
where the Caputo fractional derivative of order $\alpha$  is defined by
$D_{t}^{\alpha}y(t):=\frac{1}{\Gamma(1-\alpha)}\int_{0}^{t}(t-s)^{-\alpha}y'(s)ds$ \cite{caputo}. 
The solution is given by $y(t)=y_{0}E_{\alpha, 1}(\lambda t^{\alpha})$, where $E_{\alpha, \beta}(t)$ is the Mittag-Leffler funciton. 
When $\lambda<0$, we have $y(t)\sim Ct^{-\alpha}$ by the asymptotic expansion formula of Mittag-Leffler funciton. 
This long-term algebraic decay rate is commonly referred to as Mittag-Leffler stability. 
Another surprising result is presented in \cite{uniq-exis12}, which shows that for the fractional ODE: $D_{t}^{\alpha}y(t)=\lambda y^{\gamma}(t)$ with $y_0>0, \gamma>0$ and $\lambda<0$, the solution has the asymptotic decay rate $y(t)\sim Ct^{-\frac{\alpha}{\gamma}}$. 

Once we have the above two results and combine appropriate energy methods, we can obtain the long-term decay rate estimation of the standard $L^s(\Omega)$-norm $(s>1)$ for the solution of linear or some nonlinear anomalous diffusion model. See for example \cite{decay1, long-decay2, uniq-exis12, decay-extin}.
Currently, most of these analyses are focused on constant coefficient equations.
However, the anomalous diffusion problems can be both constant coefficients and variable coefficients. 
Compared with constant coefficient equations, the variable coefficient equations are appropriate for modeling of anomalous diffusive processes in turbulent media. For instance, modeling sub-diffusive transport inside a fractured porous media \cite{porous-media}. 
Diffusion representation of tracer particles in heterogeneous biological ambient fluids \cite{sub-diff1}. 
In particular, the application of the time dependence coefficient is a new perspective of coupling the spatial and temporal variables, which avoids the divergence caused by the long tail of the temporal probability distribution at a fixed time \cite{sub-diff3}.
For more information on the application of variable coefficients to more complex physical and biological systems, see also \cite{sub-diff2,sub-diff3,sub-diff4}.

The aim of this paper is to study both theoretically and numerically the long time Mittag-Leffler stability of the solutions for a class of nonlinear anomalous diffusion models with time-dependent coefficients 
\begin{equation}\label{pde_f}
   \left\{\begin{aligned}
    &D_{t}^{\alpha}u(t,x)+\nu t^{\beta}\mathcal{N}(u(t,x))=f(t,x)~~for~t>0,~x\in \Omega,\\
    &u(0,x)=u_{0}(x)\geq 0~~for~x\in \Omega,\\
    &u(t,x)=0~~~for~t>0,~x\in \partial\Omega,
   \end{aligned}\right.
\end{equation}
where $0<\alpha<1$, $\beta>-\alpha$, $\nu>0$, and $\Omega\subset \mathbb{R}^{d}~(d\geq1)$ is a bounded domain with smooth boundary $\partial\Omega$.
And $\mathcal{N}(u)$ can be linear or nonlinear operator, some typical examples including the Laplace operator $\mathcal{N}(u)=-\Delta u$; the $p$-Laplace operator
$\mathcal{N}(u)=-\Delta_{p} u$ with $\Delta_{p} u=\mathrm{div}(|\nabla u|^{p-2}\nabla u)$ and the mean curvature operator $\mathcal{N}(u)=-\mathrm{div}\Big(\frac{\nabla u}{\sqrt{1+|\nabla u|^{2}}}\Big)$. 
Following \cite{decay1,uniq-exis10},we assume that $\mathcal{N}(u)$ satisfies the following structural assumption 
    \begin{equation}\label{str-ass1}
        \|w(t,\cdot)\|_{L^{s}(\Omega)}^{s-1+\gamma}\leq C_{N}\int_{\Omega}w^{s-1}(t,x)\mathcal{N}(w(t,x))dx~\text{ for }~w\in L^{s}(\Omega),
    \end{equation}
where $s\in (1,\infty)$, $\gamma\in (0,+\infty)$ and $C_{N}>0$. 
One can verify that the above example of $\mathcal{N}(u)$ indeed satisfies the elliptical structure assumption \cite{decay1,  Decay-time-coef}.
The non-zero source terms $f(t,x)$ and satisfies the decay estimate
\begin{equation}\label{f_esti}
    \|f(t)\|_{L^{s}(\Omega)}\leq \frac{K t^{\beta}}{(t+1)^{\alpha+\beta}}~~\forall t>0,~\forall s\in (1,\infty),~K>0.
\end{equation}

Let $u$ be the solution of \eqref{pde_f}. Under the above assumptions, we can establish the long-term estimation of $\|u(t, \cdot)\|_{L^{s}(\Omega)}$, i.e. Mittag-Leffler stability, by combining the energy method with the long-term estimation of time fractional ODEs.
For the homogeneous linear constant coefficient equation with $\beta=0, f=0$ and $\mathcal{N}(u)=-\Delta u$, we can apply standard eigenvalue and eigenfunction pairs  methods to obtain that $\|u(t, \cdot)\|_{L^{s}(\Omega)}\sim C t^{-\alpha}$. 
For the constant coefficient nonlinear equation with $\beta=0$ and $\mathcal{N}(u)$ satisfying \eqref{str-ass1}, it is shown that $\|u(t, \cdot)\|_{L^{s}(\Omega)}\sim C t^{-\alpha/\gamma}$; see \cite{uniq-exis10,decay1,decay-extin}. And for the homogeneous varying coefficient nonlinear equation with $f=0$ and $\mathcal{N}(u)$ satisfying \eqref{str-ass1}, the asymptotic estimate $\|u(t, \cdot)\|_{L^{s}(\Omega)}\sim C t^{-(\alpha+\beta)/\gamma}$ for weak solution is obtained in \cite{Decay-time-coef}, from which can see from that the time-dependent factor $t^{\beta}$  has a direct impact on the long-term decay rate of the solution. 

At present, most numerical methods for time fractional equations focus on the reduction in numerical convergence rate caused by weak singularity of the solution near the initial moment. Various effective techniques have been developed to restore high-order convergence of numerical methods; see \cite{liao2019discrete, Kop19, StyL1survey, jin2023numerical}.
From the perspective of the structure-preserving algorithm, we naturally hope that the numerical solution can preserve the long-term optimal decay rate of the sub-diffusion equation, that is, establish numerical Mittag-Leffler stability. At present, there are relatively few research results on qualitative analysis of numerical solutions to anomalous diffusion equations, especially long-term behavior analysis. 
For the linear time fractional evolutional equation, the numerical ML stability, i.e. $\|u_{n}\|_{L^{s}(\Omega)}\sim Ct_{n}^{-\alpha}$ is established in \cite{Mittag-leffler-ODE} through singularity analysis of the generating function of the numerical solution, where $u_{n}$ is the numerical approximation of $u$ at $t_{n}$.
For the constant coefficient nonlinear equation with $\beta=0$ and $\mathcal{N}(u)$ satisfying \eqref{str-ass1}, the numerical ML stability as $\|u_{n}\|_{L^{s}(\Omega)}\sim Ct_{n}^{-\alpha/\gamma}$ was recently established in \cite{CM-preserving}, using the discrete comparison principle and the method of constructing fine discrete sup and sub solutions.
The key to success here lies in a class of time fractional derivative structure-preserving discrete schemes, namely the $\mathcal{CM}$-preserving schemes developed in \cite{com-mono}. The discrete coefficients of this type of numerical method have many excellent properties.

The objective of this article is twofold. 
\begin{itemize}
\item We derive the long-term decay estimate for the solution of  the general non-homogeneous variable coefficient nonlinear equation \eqref{pde_f} with $\beta\neq0, f\neq0$ and $\mathcal{N}(u)$ satisfying \eqref{str-ass1} and prove that $\|u(t, \cdot)\|_{L^{s}(\Omega)}\sim C t^{-(\alpha+\beta)/\gamma}$.  Our results extend the work of \cite{Decay-time-coef, decay-extin}.

\item We establish the numerical ML stability for \eqref{pde_f} based on the $\mathcal{CM}$-preserving schemes and show that $\|u_{n}\|_{L^{s}(\Omega)}\sim C t_{n}^{-(\alpha+\beta)/\gamma}$. Firstly, we prove that the $\mathcal{CM}$-preserving schemes also satisfies the discrete comparison principle for variable coefficient equations. Then, applying this result, we construct discrete sup and sub solutions with variable coefficients to obtain the expected optimal numerical decay rate estimate. Our work in this article extend the results in \cite{CM-preserving} from homogeneous constant coefficient equations to non-homogeneous, variable coefficient  equations.
\end{itemize}

The rest of this article is organized as follows. In Section \ref{sec1}, we derive the long time decay rate of continuous solutions for \eqref{pde_f} and provide detailed proofs of the estimates.
In Section \ref{sec2}, we recall $\mathcal{CM}$-preserving schemes on uniform mesh $t_{n}=n\tau$ with time step size $\tau>0$ and their main properties.
In Section \ref{sec3}, we establish the discrete numerical ML stability for \eqref{pde_f}.  We first derive the discrete comparison principle of $\mathcal{CM}$-preserving schemes for a class of time-dependent fractional ODEs. Then, by constructing discrete sub and sup solution we show the numerical ML stability for fractional non-homogeneous, variable coefficient ODEs. Then, by establishing a variable coefficient discrete energy inequality, we obtain the estimate of the long-term decay rate of the numerical solution of \eqref{pde_f}.
Finally, we present  several typical numerical examples to illustrate and verify the optical numerical decay rate in Section \ref{sec5}.

\section{Decay estimate for nonlinear models with time-dependent coefficients}\label{sec1}
Consider the time-dependent coefficients fractional ODE, which highlight the influence of $t^{\beta}$:
\begin{equation}\label{fode}
    D_{t}^{\alpha}y(t)+\nu t^{\beta}y^{\gamma}(t)=0 \text{ with } y(0)=y_{0}>0,
\end{equation}
where $0<\alpha<1$, $\beta>-\alpha$, $\gamma>0$ and $\nu>0$. 
The decay result concerning \eqref{fode} reads as follows.

\begin{lemma}\label{l-c-decay}\cite{Decay-time-coef}
Let $y\in C(\mathbb{R}_{+})$ be the solution of \eqref{fode}. Then there exist $C_{1}, C_{2}>0$ such that
\begin{equation}\label{c-decay}
    \frac{C_{1}}{1+t^{\frac{\alpha+\beta}{\gamma}}}\leq y(t)\leq \frac{C_{2}}{1+t^{\frac{\alpha+\beta}{\gamma}}}~~for~t\geq 0.
\end{equation}
\end{lemma}

Note that when $\beta=0$, the fractional ODE \eqref{fode} is reduced to the power nonlinear equation and its solution has decay rate $y(t)\sim Ct^{-\alpha/\gamma}$ \cite{uniq-exis10}. Further, when $\beta=0$ and $\gamma=1$,  it becomes fractional linear ODE and the solution has the decay rate $y(t)\sim t^{-\alpha}$.
Our first result below is to extend the estimation of the long-term decay rate of homogeneous equation \eqref{fode} to non-homogeneous cases with source terms.

The proof the Mittag-Leffler stability of fractional ODEs in \cite{uniq-exis10, Decay-time-coef} is mainly based on the comparison principle 
and the construction of appropriate barrier functions.

\begin{lemma}[Fractional comparison principle]\cite{uniq-exis10}\label{lem-conti}
Let $T>0$ and $g\in C(\mathbb{R})$. Assume that $g$ is nondecreasing. Suppose that $u$, $w\in H_{1}^{1}([0,T])$ satisfy $u(0)\leq w(0)$ and
\begin{equation*}
    \begin{aligned}
        D_{t}^{\alpha}(u-u(0))+g(u)&\leq 0,~a.a.~t\in (0,T),\\
        D_{t}^{\alpha}(w-w(0))+g(w)&\geq 0,~a.a.~t\in (0,T).
    \end{aligned}
\end{equation*}
Then $u(t)\leq w(t)$ for all $t\in[0,T]$.   
\end{lemma}

It is noted that in \eqref{fode}, the nonlinear function $g(y):=\nu t^{\beta}y^{\gamma}$ is positive and nondecreasing, which allows the utilization of Lemma \ref{lem-conti} to prove Lemma \ref{l-c-decay}.  Now we extend Lemma  \ref{l-c-decay} to non-homogeneous cases with source terms.

\begin{theorem}\label{t-nonli-f}
   Consider the fractional ODE: $D_{t}^{\alpha}\varphi(t)+\nu t^{\beta}\varphi^{\gamma}(t)=f(t)$ with $\varphi(0)=\varphi_{0}>0$ and assume that $f$ satisfies $|f|<Kt^{\beta}/(t+1)^{\alpha+\beta}$ for some $K>0$. 
    Then there exists constants $\breve{C}_{1}$, $\breve{C}_{2}>0$ such that the solution $\varphi$ satisfies
    \begin{equation}\label{c-decay-f}
    \frac{\breve{C}_{1}}{1+t^{\frac{\alpha+\beta}{\gamma}}}\leq \varphi(t)\leq \frac{\breve{C}_{2}}{1+t^{\frac{\alpha+\beta}{\gamma}}}~~for~t\geq 0.
\end{equation}
\end{theorem}

Our main idea of proof is as follows: by constructing new fine upper and lower solutions, we can control the source function $f$ and obtain inequality of the corresponding homogeneous equation, which can be applied to Lemma \ref{lem-conti}.

\begin{proof}
    By the assumption $|f|<Kt^{\beta}/(t+1)^{\alpha+\beta}$, we have 
    \begin{equation}\label{sol-f-ineq}
        \frac{-Kt^{\beta}}{(t+1)^{\alpha+\beta}}\leq D_{t}^{\alpha}\varphi(t)+\nu t^{\beta}\varphi^{\gamma}(t)\leq \frac{Kt^{\beta}}{(t+1)^{\alpha+\beta}}.
    \end{equation}
   Define $\varphi_{0}^{\gamma}=\max \left\{ K/\nu, (4\nu\Gamma(1-\alpha))^{\gamma/(1-\gamma)} \right\}$. \\
    
\textbf{I. Construct a subsolution.} Consider the function 
    \begin{equation}
      u(t)= \left\{\begin{aligned}
            &\varphi_{0}- \frac{ 2\nu\varphi_{0}^{\gamma}\Gamma(1+\beta) } {\Gamma(1+\alpha+\beta)} t^{\alpha+\beta}~~for~t\in[0,\epsilon],\\
            &C_{*}\varphi_{0}\epsilon^{\frac{\alpha+\beta}{\gamma}}t^{-\frac{\alpha+\beta}{\gamma}}~~for~t> \epsilon,
        \end{aligned}\right.
    \end{equation}
    where 
    \begin{equation*}
        \epsilon=\left( \frac{(1-C_{*})\varphi_{0}^{1-\gamma}\Gamma(1+\alpha+\beta)}{2\nu \Gamma(1+\beta)} \right)^{1/(\alpha+\beta)}, 
        ~C_{*}=\min \left\{ \left(\frac{\Gamma(1+\beta)}{\Gamma(1+\alpha+\beta)\Gamma(1-\alpha)}\right)^{\frac{1}{\gamma}}, 2^{-1} \right\}.
    \end{equation*}
    
Using the definition of $u(t)$, it is easy to find that $v$ is monotonically decreasing. 
Now we show that $u(t)$ is a sub solution for $\varphi$ over the interval $[0,\epsilon]\cup (\epsilon,\infty)$. 
For $t\in [0,\epsilon]$, by noting $\varphi_{0}^{\gamma}\geq K/\nu$ and $\nu,~\varphi_{0}>0$, we have 
    \begin{equation*}\label{v-est-1}
        \begin{aligned}
            D_{t}^{\alpha}u(t) +\nu t^{\beta}u^{\gamma}(t)&=-2\nu t^{\beta}\varphi_{0}^{\gamma}+\nu t^{\beta}\left(\varphi_{0}- \frac{2\nu\varphi_{0}^{\gamma}\Gamma(1+\beta)}{\Gamma(1+\alpha+\beta)}t^{\alpha+\beta }\right)^{\gamma}\\
            &\leq-\nu t^{\beta}\varphi_{0}^{\gamma}\leq-t^{\beta}K\leq-\frac{t^{\beta}K}{(t+1)^{\alpha+\beta}}.
        \end{aligned}
    \end{equation*}

    On the other hand, for $t>\epsilon$ we have $v'\leq 0$.  Therefor, 
    \begin{equation*}
        \begin{aligned}
            D_{t}^{\alpha}u(t) +\nu t^{\beta}u^{\gamma}(t)&\leq\frac{1}{\Gamma(1-\alpha)}\int_{0}^{\epsilon}(t-s)^{-\alpha}v'(s)ds+\nu t^{\beta}C_{*}^{\gamma}\varphi_{0}^{\gamma}\epsilon^{\alpha+\beta}t^{-(\alpha+\beta)}\\
            &\leq -\frac{2\nu\Gamma(1+\beta)}{\Gamma(1+\alpha+\beta)\Gamma(1-\alpha)}\varphi_{0}^{\gamma}t^{-\alpha}\epsilon^{\alpha+\beta}+\nu t^{-\alpha}C_{*}^{\gamma}\varphi_{0}^{\gamma}\epsilon^{\alpha+\beta}\\
            &=-\nu t^{-\alpha}\varphi_{0}^{\gamma} \epsilon^{\alpha+\beta}\left(\frac{2\Gamma(1+\beta)}{\Gamma(1+\alpha+\beta)\Gamma(1-\alpha)}-C_{*}^{\gamma}\right).
        \end{aligned}
    \end{equation*}  
   By the definitions of $\epsilon$, $C_{*}$ and noting that $\varphi_{0}>[4\nu \Gamma(1-\alpha)]^{1/(1-\gamma)}$, we confirm that 
    \begin{equation*}
        \epsilon^{\alpha+\beta}\left(\frac{2\Gamma(1+\beta)}{\Gamma(1+\alpha+\beta)\Gamma(1-\alpha)}-C_{*}^{\gamma}\right) \geq \epsilon^{\alpha+\beta}\frac{\Gamma(1+\beta)}{\Gamma(1+\alpha+\beta)\Gamma(1-\alpha)}=\frac{(1-C_{*})\varphi_{0}^{1-\gamma}}{2\nu \Gamma(1-\alpha)}\geq\frac{\varphi_{0}^{1-\gamma}}{4\nu \Gamma(1-\alpha)}\geq1,
    \end{equation*}
    then 
    \begin{equation}\label{v-ineq}
        D_{t}^{\alpha}u(t) +\nu t^{\beta}u^{\gamma}(t)\leq-\nu t^{-\alpha}\varphi_{0}^{\gamma}=-\nu t^{\beta}t^{-(\alpha+\beta)}\varphi_{0}^{\gamma}\leq-\frac{t^{\beta}K}{(t+1)^{\alpha+\beta}}.
    \end{equation}
    Hence, combining \eqref{sol-f-ineq}, \eqref{v-est-1} and \eqref{v-ineq}, we have 
    \begin{equation*}
        D_{t}^{\alpha}u(t) +\nu t^{\beta}u^{\gamma}(t)\le -\frac{Kt^{\beta}}{(t+1)^{\alpha+\beta}}  \leq D_{t}^{\alpha}\varphi(t) +\nu t^{\beta}\varphi^{\gamma}(t) \text{ for all } t>0.
    \end{equation*}

    \textbf{II. Construct a supersolution.} Define $t_{0}>0$ by means of 
    \begin{equation*}
        t_{0}:= \left[\frac{\varphi_{0}^{1-\gamma}2^{\alpha+1}}{\nu}\left(\frac{1}{\Gamma(1-\alpha)}+\frac{\alpha+\beta}{\gamma\Gamma(2-\alpha)}\right)\right]^{1/(\alpha+\beta)}.
    \end{equation*}
    Let us consider the monotonically decreasing function 
    \begin{equation}
       w(t)= \left\{\begin{aligned}
            &\varphi_{0}~~for~t\in[0,t_{0}],\\
            &\varphi_{0}t_{0}^{\frac{\alpha+\beta}{\gamma}}t^{-\frac{\alpha+\beta}{\gamma}}~~for~t> t_{0}.
        \end{aligned}\right.
    \end{equation}
    
    For $0<t<t_{0}$, we have
    \begin{equation}\label{w-est-1}
        D_{t}^{\alpha}w(t) +\nu t^{\beta}w^{\gamma}(t)=\nu t^{\beta}\varphi_{0}^{\gamma}>Kt^{\beta}\geq \frac{Kt^{\beta}}{(t+1)^{\alpha+\beta}}.
    \end{equation}
    
    Next, observe that for $t\in [t_{0},2t_{0}]$, we may thus estimate as follows
    \begin{equation*}
        \begin{aligned}
            D_{t}^{\alpha}w(t) +\nu t^{\beta}w^{\gamma}(t)&=-\frac{\varphi_{0}}{\Gamma(1-\alpha)}\frac{\alpha+\beta}{\gamma}t_{0}^{\frac{\alpha+\beta}{\gamma}}\int_{t_{0}}^{t}(t-s)^{-\alpha}s^{-\frac{\alpha+\beta}{\gamma}-1}ds+\nu t^{\beta}\varphi_{0}^{\gamma}t_{0}^{\alpha+\beta}t^{-(\alpha+\beta)}\\
            &\geq -\frac{\varphi_{0}}{\Gamma(2-\alpha)}\frac{\alpha+\beta}{\gamma}t_{0}^{-\alpha}+\frac{\nu}{2} \varphi_{0}^{\gamma}t_{0}^{\alpha+\beta}t^{-\alpha}+\frac{\nu}{2} t^{\beta}\varphi_{0}^{\gamma}t_{0}^{\alpha+\beta}t^{-(\alpha+\beta)}\\
            &\geq \left(-\frac{\alpha+\beta}{\gamma}\frac{2^{\alpha}\varphi_{0}}{\Gamma(2-\alpha)}+\frac{\nu}{2} \varphi_{0}^{\gamma}t_{0}^{\alpha+\beta}\right)t^{-\alpha}+\frac{\frac{1}{2}Kt_{0}^{\alpha+\beta}t^{\beta}}{(t+1)^{\alpha+\beta}}.\\
        \end{aligned}
    \end{equation*}
 It follows from the definition $t_0$ that $\left(-\frac{\alpha+\beta}{\gamma}\frac{2^{\alpha}\varphi_{0}}{\Gamma(2-\alpha)}+\frac{\nu}{2} \varphi_{0}^{\gamma}t_{0}^{\alpha+\beta}\right)\geq 0$ and $\frac{\frac{1}{2}Kt_{0}^{\alpha+\beta}t^{\beta}}{(t+1)^{\alpha+\beta}}\geq \frac{Kt^{\beta}}{(t+1)^{\alpha+\beta}}$. 
 Hence, we have $D_{t}^{\alpha}w(t) +\nu t^{\beta}w^{\gamma}(t)\geq \frac{Kt^{\beta}}{(t+1)^{\alpha+\beta}}$ for $t\in [t_{0},2t_{0}]$.

    Finally, suppose that $t>2t_{0}$, we obtain
     \begin{equation*}
        \begin{aligned}
            D_{t}^{\alpha}w(t) +\nu t^{\beta}w^{\gamma}(t)&=-\frac{\varphi_{0}(\alpha+\beta)}{\gamma\Gamma(1-\alpha)}t_{0}^{\frac{\alpha+\beta}{\gamma}}t^{-\alpha-\frac{\alpha+\beta}{\gamma}}\int_{\frac{t_{0}}{t}}^{1}(1-s)^{-\alpha}s^{-\frac{\alpha+\beta}{\gamma}-1}ds+\nu t^{\beta}\varphi_{0}^{\gamma}t_{0}^{\alpha+\beta}t^{-(\alpha+\beta)}\\
            &=-\frac{\varphi_{0}(\alpha+\beta)}{\gamma\Gamma(1-\alpha)}t_{0}^{\frac{\alpha+\beta}{\gamma}}t^{-\alpha-\frac{\alpha+\beta}{\gamma}}\left(\int_{\frac{t_{0}}{t}}^{\frac{1}{2}}\cdots+\int_{\frac{1}{2}}^{1}\cdots\right)+\nu t^{\beta}\varphi_{0}^{\gamma}t_{0}^{\alpha+\beta}t^{-(\alpha+\beta)}\\
            &\geq-\varphi_{0}t_{0}^{\frac{\alpha+\beta}{\gamma}}t^{-\alpha-\frac{\alpha+\beta}{\gamma}}\left(\frac{1}{\Gamma(1-\alpha)}2^{\alpha}t_{0}^{-\frac{\alpha+\beta}{\gamma}}t^{\frac{\alpha+\beta}{\gamma}}+\frac{\alpha+\beta}{\gamma}2^{\frac{\alpha+\beta}{\gamma}+\alpha}\frac{1}{\Gamma(2-\alpha)}\right)+\nu t^{\beta}\varphi_{0}^{\gamma}t_{0}^{\alpha+\beta}t^{-(\alpha+\beta)}\\
            &\geq \left(-\frac{2^{\alpha}\varphi_{0}}{\Gamma(1-\alpha)}-\frac{\alpha+\beta}{\gamma}\frac{2^{\alpha}\varphi_{0}}{\Gamma(2-\alpha)}+\frac{\nu}{2} \varphi_{0}^{\gamma}t_{0}^{\alpha+\beta}\right)t^{-\alpha}+\frac{\frac{1}{2}Kt_{0}^{\alpha+\beta}t^{\beta}}{(t+1)^{\alpha+\beta}}.
        \end{aligned}
    \end{equation*}    
    By definition of $t_{0}$,  we also have 
    $\left(-\frac{2^{\alpha}\varphi_{0}}{\Gamma(1-\alpha)}-\frac{\alpha+\beta}{\gamma}\frac{2^{\alpha}\varphi_{0}}{\Gamma(2-\alpha)}+\frac{\nu}{2} \varphi_{0}^{\gamma}t_{0}^{\alpha+\beta}\right)\geq 0$, which yields
    $D_{t}^{\alpha}w(t) +\nu t^{\beta}w^{\gamma}(t)\geq \frac{Kt^{\beta}}{(t+1)^{\alpha+\beta}}$ for $t>2 t_{0}$.
In summary, we get that
    \begin{equation*}
        D_{t}^{\alpha}\varphi(t) +\nu t^{\beta}\varphi^{\gamma}(t)\le \frac{Kt^{\beta}}{(t+1)^{\alpha+\beta}} \leq D_{t}^{\alpha}w(t) +\nu t^{\beta}w^{\gamma}(t) \text{ for all }t\geq0.
    \end{equation*}

The fractional comparison principle of Lemma \ref{lem-conti} shows that $u(t)\le \varphi(t)\leq w(t)$ for all $t\ge 0$, which gives the required decay rate of equation \eqref{c-decay-f}.
\end{proof}

As an immediate application of Theorem \ref{t-nonli-f}, we can derive the long-term decay rate for $\|u(x, t)\|_{L^{s}(\Omega)}$ for the solution of the time fractional PDE model \eqref{f_esti} by combination of appropriate energy estimate.

\begin{corollary}\label{c-decay-f-N}
   Let $s>1$, $T>0$ and $\Omega\subset\mathbb{R}^{N}$. Let $u(t,x)\in L^{s}([0,T];L^{s}(\Omega))$ be a nonnegative weak solution of the model \eqref{pde_f}. 
   Assume that the nonlinear operator $\mathcal{N}[u]$ satisfies the structural assumption \eqref{str-ass1} and $f(t,x)$ satisfies decay estimate \eqref{f_esti}. 
    Then there exists a positive constant $C$ such that
    \begin{equation}\label{Dpde}
        \|u\|_{L^{s}(\Omega)}(t)\leq \frac{C}{1+t^{\frac{\alpha+\beta}{\gamma}}}~\text{ for all }~t>0.
    \end{equation}
\end{corollary}
\begin{proof}
    Multiplying the equation $D_{t}^{\alpha}u(t,x)+ \nu t^{\beta}\mathcal{N}(u(t,x))=f(t,x)$ by $|u|^{s-2}u$ and integrating it over $\Omega$ to find  
    \begin{equation*}
      \int_{\Omega}|u|^{s-2}uD_{t}^{\alpha}udx+ \int_{\Omega}\nu t^{\beta}\mathcal{N}(u)|u|^{s-2}udx=\int_{\Omega}f(t,x)|u|^{s-2}udx.
    \end{equation*}
On one hand, for the nonnegative weak solution $u(t,x)\geq 0$, we have \cite[Page 221, Corollary 3.1]{uniq-exis10}
     \begin{equation}\label{dt-inq}
        \|u(t,\cdot)\|_{L^{s}(\Omega)}^{s-1}D_{t}^{\alpha}\|u(t,\cdot)\|_{L^{s}(\Omega)}\leq \int_{\Omega}|u(t,x)|^{s-2}u(t,x) \cdot D_{t}^{\alpha}u(t,x)dx.
    \end{equation}     
Thus, by \eqref{dt-inq} and the structural assumption \eqref{str-ass1} we have
    \begin{equation*}
        \|u(t,\cdot)\|_{L^{s}(\Omega)}^{s-1}D_{t}^{\alpha}\|u(t,\cdot)\|_{L^{s}(\Omega)} +\nu t^{\beta}\|u(t,\cdot)\|_{L^{s}(\Omega)}^{s+\gamma-1}\leq\int_{\Omega}f(t,x)|u|^{s-2}udx.
    \end{equation*}
According to H$\mathrm{\Ddot{o}}$lder inequality and decay estimate \eqref{f_esti}, one has 
\begin{equation*}
    \int_{\Omega}f(t,x)|u|^{s-2}udx\leq \|f(t)\|_{L^{s}(\Omega)}\|u(t,\cdot)\|_{L^{s}(\Omega)}^{s-1}\leq \frac{Kt^{\beta}\|u(t,\cdot)\|_{L^{s}(\Omega)}^{s-1}}{(t+1)^{\alpha+\beta}}.
\end{equation*}
Hence, we get that
    \begin{equation*}
        \|u(t,\cdot)\|_{L^{s}(\Omega)}^{s-1}D_{t}^{\alpha}\|u(t,\cdot)\|_{L^{s}(\Omega)} +\nu t^{\beta}\|u(t,\cdot)\|_{L^{s}(\Omega)}^{s+\gamma-1}\leq\frac{Kt^{\beta}\|u(t,\cdot)\|_{L^{s}(\Omega)}^{s-1}}{(t+1)^{\alpha+\beta}}.
    \end{equation*}
    
Clearly, if  $\|u(t,\cdot)\|=0$, the required inequality \eqref{Dpde} holds. On the other hand, if  $\|u(t,\cdot)\|>0$, we have
    \begin{equation*}
        D_{t}^{\alpha}\|u\|_{L^{s}(\Omega)}(t) +\nu t^{\beta}\|u\|_{L^{s}(\Omega)}^{\gamma}(t)\leq\frac{Kt^{\beta}}{(t+1)^{\alpha+\beta}}.
    \end{equation*}   
 Let $\varphi(t):=\|u(t, \cdot)\|_{L^{s}(\Omega)}$ and by Theorem \ref{t-nonli-f}, the decay estimate \eqref{Dpde} follows.
\end{proof}

\section{$\mathcal{CM}$-preserving numerical methods}\label{sec2}
From the perspective of structural preservation algorithms, it is natural to develop structural preservation algorithms for time fractional order equations, so that numerical solutions can accurately preserve the long-term optimal decay rate of the solution in Lemma \ref{t-nonli-f} for nonlinear fractional ODEs and in Corollary \ref{c-decay-f-N} for nonlinear fractional PDEs. That is to develop numerical Mittag-Leffler stability for the model \eqref{pde_f}.

The Caputo fractional derivative $D_{t}^{\alpha}(y(t))$ has a convolutional structure,  so we consider that its numerical approximation also has a corresponding discrete convolution structure, which takes the form of
\begin{equation}\label{dis-caputo}
    D_{\tau}^{\alpha}(y_{n})=\frac{1}{\tau^{\alpha}}\sum_{k=0}^{n}\omega_{k}(y_{n-k}-y_{0})=\frac{1}{\tau^{\alpha}}\Big(\delta_{n}y_{0}+\sum_{k=0}^{n-1}\omega_{k}y_{n-k}\Big)~~for~n=1,2,\dots
\end{equation}
where $\delta_{n}:=-\sum_{k=0}^{n-1}\omega_{k}$ and $y_{k}$ denotes the numerical approximation of $y(t_{k})$ on the uniform meshes $t_n=n\tau$ with step size $\tau>0$. The weight coefficients $\{\omega_{k}\}_{k=0}^{\infty}$ can be computed explicitly. 
The two most popular methods are Gr\"unwald-Letnikov scheme and L1 method, both of which can be seen as fractional extensions of the classical backward Euler scheme.
\begin{example}\label{exp1}
  For L1 method, the weight $\omega_{k}$ satisfies 
    \[
    \omega_{0}=\frac{1}{\Gamma(2-\alpha)},~\omega_{k}=\frac{1}{\Gamma(2-\alpha)} \left[ (k+1)^{1-\alpha}-2k^{1-\alpha}+(k-1)^{1-\alpha} \right] \text{ for } k=1,2,\dots.
    \]
and consequently $\delta_{n}=-\sum_{k=0}^{n-1}\omega_{k}=[(n-1)^{1-\alpha}-n^{1-\alpha}]/\Gamma(2-\alpha)$.
\end{example}

\begin{example}\label{exp2}
   For the Gr\"unwald-Letnikov scheme, the weights $\omega_{k}$ are the $k$-th coefficients of the generating function $F_{\omega}(z)=(1-z)^{\alpha}=\sum_{k=0}^{\infty}\omega_{k}z^k$. Therefore, 
\[
\omega_{0}=1,~\omega_{k}=\left( 1-\frac{\alpha+1}{k} \right) \omega_{k-1} \text{ for } k=1,2,\dots.
\]
\end{example}

By using Riemann-Liouville fractional integration, the initial value problem of the Caputo derivative can be transformed into the corresponding Volterra integral equation, which also has a convolutional structure and its kernel function is $k_{\alpha}(t):=t^{\alpha-1}/\Gamma(\alpha)$ for $t>0$. The standard kernel function $k_{\alpha}(t)$ is a very important class of completely monotonic ($\mathcal{CM}$) functions 
\footnote{A function $g(t): (0, \infty)\to \mathbb{R}$ is said to be completely monotone if $(-1)^n g^{(n)}(t)\geq 0$ for all $n\geq0$ and $t>0$. }. Based on the basic idea of preserving structure algorithm, the $\mathcal{CM}$-preserving numerical methods for Riemann-Liouville fractional integration is developed in \cite{com-mono}. The new concept and framework of $\mathcal{CM}$-preserving numerical methods has many excellent properties in the discrete coefficients.  This kind of numerical methods has recently been successfully applied to a class of nonlinear anomalous diffusion models
\cite{CM-preserving}.
It is proved in \cite{com-mono} that for $\mathcal{CM}$-preserving numerical methods the weights $\omega_{k}$ and $\delta_{n}$ have the convenient properties
\begin{equation}\label{coe-pro1}
    \left\{\begin{aligned}
        &(i)~\omega_{0}>0,~\omega_{1}<\omega_{2}<\cdots<\omega_{n}<0~~~~(monotonicity);\\
        &(ii)~\sum_{k=0}^{\infty}\omega_{k}=0~~~~(conservation);\\
        &(iii)~\omega_{n}=O(n^{-1-\alpha}),~\delta_{n}=-\sum_{k=0}^{n-1}\omega_{k}=O(n^{-\alpha})~~~~(uniform~decay~rate).
    \end{aligned}\right.
\end{equation}
These properties imply that there exist constants $0<C_{3}\leq C_{4}$, which are independent of $n$, sunch that
\begin{equation}\label{coe-pro2}
    \frac{C_{3}}{n^{1+\alpha}}\leq|\omega_{n}|\leq\frac{C_{4}}{n^{1+\alpha}},~~\frac{C_{3}}{n^{\alpha}}\leq|\delta_{n}|\leq\frac{C_{4}}{n^{\alpha}},~~\frac{C_{3}}{n^{\alpha}}\leq\sum_{k=n}^{\infty }|\omega_{k}|\leq\frac{C_{4}}{n^{\alpha}}~for~n=1,2,\dots.
\end{equation}

It is easy to verify the Gr\"unwald-Letnikov scheme and L1 method in the above two examples are $\mathcal{CM}$-preserving
and their coefficients meet the above properties. See more details in \cite{com-mono}.

\section{Numerical Mittag-Leffler stability} \label{sec3}
In the section, we study the numerical Mittag-Leffler stability for $\mathcal{CM}$-preserving numerical methods for nonlinear sub-diffusion models with variable coefficients. Specifically, we prove that for the $\mathcal{CM}$-preserving schemes, its numerical solution has a long-term optimal decay rate that is completely consistent with the corresponding continuous equation. Our proof mainly consists of the following three steps. Firstly, we establish the discrete comparison principle for variable coefficient equations. Then we apply this result to homogeneous equations with variable coefficients and establish the Mittag-Leffler stability of numerical solutions. Finally, for general non-homogeneous equations with source terms, we also establish the numerical Mittag-Leffler stability.

\subsection{Discrete fractional comparison principle}
We have the following discrete fractional comparison principle, which extend the results of \cite{com-mono}. 

\begin{lemma}\label{dis-com-prin}
   Let $D_{\tau}^{\alpha}$ be the $\mathcal{CM}$-preserving discrete operator defined in \eqref{dis-caputo}. Assume the function $f(\cdot)$ is nondecreasing and $\nu>0$. Suppose that the sequences $\{u_{j}\}_{j=0}^{\infty}$, $\{y_{j}\}_{j=0}^{\infty}$,$\{w_{j}\}_{j=0}^{\infty}$ satisfy $u_{0}\leq y_{0}\leq w_{0}$ and
 \[
 D_{\tau}^{\alpha}(u_{n})+\nu t_{n}^{\beta}f(u_{n})\leq 0,~~D_{\tau}^{\alpha}(y_{n})+\nu t_{n}^{\beta}f(y_{n})= 0,~~D_{\tau}^{\alpha}(w_{n})+\nu t_{n}^{\beta}f(w_{n})\geq 0~~for~n\geq 1.
 \]
  Then $u_{n}\leq y_{n}\leq w_{n}$ for all $n=1,2,\dots$. We call $\{u_{j}\}_{j=0}^{\infty}$  a discrete subsolution for $\{y_{j}\}_{j=0}^{\infty}$ and $\{w_{j}\}_{j=0}^{\infty}$ a discrete supsolution.
\end{lemma}

\begin{proof}    
Define the sequence  $\xi_{n}:=u_{n}-y_{n}$ for $n\geq0$ with $\xi_{0}=u_{n}-y_{n}\leq 0$. 
Then we get that $D_{\tau}^{\alpha}(\xi_{n})\leq -\nu t_{n}^{\beta}[f(u_{n})-f(y_{n})]$.
Because $f(\cdot)$ is nondecreasing and $\nu>0$ and  $t_{n}^{\beta}>0$ for any $\beta>-\alpha$, we further get that
    \begin{align}\label{po_inq1}
    \frac{1}{\tau^{\alpha}}\left(\sum_{k=0}^{n-1}\omega_{k}\xi_{n-k}+\delta_{n}\xi_{0}\right)\leq  -\nu t_{n}^{\beta} [f(u_{n})-f(y_{n})]\leq 0.
   \end{align}
Define the indicator function
    \begin{equation}
    \chi_{(\xi_{n}\geq0)}=\left\{
        \begin{aligned}
            &1~~for~~\xi_{n}\geq0;\\
            &0~~for~~\xi_{n}<0,
        \end{aligned}\right.
    \end{equation}
Multiplying the inequality \eqref{po_inq1}  by $\chi_{(\xi_{n}\geq0)}$ gives
\begin{align*}
    \frac{1}{\tau^{\alpha}}\left(\omega_{0}\xi_{n}\chi_{(\xi_{n}\geq0)}+\sum_{k=1}^{n-1}\omega_{k}\xi_{n-k}\chi_{(\xi_{n}\geq0)}+\delta_{n}\xi_{0}\chi_{(\xi_{n}\geq0)}\right)\leq  -\nu t_{n}^{\beta} [f(u_{n})-f(y_{n})] \chi_{(\xi_{n}\geq0)}\leq 0.
\end{align*}
It's easy to see $\xi_{k}\chi_{(\xi_{n}\geq0)}\leq\max\{\xi_{k}, 0\}$.
Let $\eta_{n}:=\max\{ \xi_{n}, 0\}$ with $\eta_{0}=0$.
Since $\omega_{k}\leq0$ for $k\geq1$ and $\delta_{n}\leq 0$,  we have
\begin{align*}
    \omega_{0}\xi_{n}\chi_{(\xi_{n}\geq0)}+\sum_{k=1}^{n-1}\omega_{k}\xi_{n-k}\chi_{(\xi_{n}\geq0)}+\delta_{n}\xi_{0}\chi_{(\xi_{n}\geq0)}\geq \omega_{0}\eta_{n}+\sum_{k=1}^{n-1}\omega_{k}\eta_{n-k}+\delta_{n}\eta_{0}.
\end{align*}
That is $0\geq\omega_{0}\eta_{n}+\sum_{k=1}^{n-1}\omega_{k}\eta_{n-k}$ for all $n\geq 1$.
For $n=1$, because of $\omega_{0}>0$ and $\omega_{0}\eta_{1}\leq 0$, one has $\eta_{1}\leq 0$. Meanwhile, we know $\eta_{1}\geq0$ from the definition of $\eta_{1}$, which yields that $\eta_{1}=0$.
Following a similar method, we can easily obtain $\eta_{n}=0$ for all $n\geq 1$. This implies that $\xi_{n}\leq 0$ for all $n\geq 1$ , i.e., $u_{n}\leq y_{n}$.
The argument for $y_{n}\leq w_{n}$ is similar.
\end{proof}

\subsection{Numerical Mittag-Leffler stability for homogeneous equations} \label{sec:32}
With the above preparation, we can now prove our first result, which is the numerical Mittag-Leffler stability of nonlinear homogeneous equations with variable coefficients. We focus here on the direct impact of the coefficient factor $t^\beta$ on the long-term optimal decay rate.

\begin{theorem}\label{uni-decay}
 For the nonlinear fractional ODE \eqref{fode}, consider the $\mathcal{CM}$-preserving scheme on uniform time meshes $\{t_{n}=n\tau\}_{n=0}^{\infty}$ given by
    \begin{equation}\label{time-dis}
         D_{\tau}^{\alpha}(y_{n})=\frac{1}{\tau^{\alpha}}\left(\delta_{n}y_{0}+\sum_{k=0}^{n-1}\omega_{k}y_{n-k}\right)=-\nu t_{n}^{\beta}y_{n}^{\gamma}~~for~n\geq 1,
    \end{equation}
    where $y_0>0$, $\beta>-\alpha$, $\gamma>0$ and $\nu>0$. Then the numerical solution $\{y_{n}\}$ of \eqref{time-dis} satisfies
    \begin{equation}\label{un-decay}
        \frac{C_{5}}{1+t_{n}^{(\alpha+\beta)/\gamma}}\leq y_{n}\leq \frac{C_{6}}{1+t_{n}^{(\alpha+\beta)/\gamma}}~~for~n=0,1,2,\dots,
    \end{equation}
    where the positive constants $C_{5}$, $C_{6}$ are independent of $n$ and $\tau$.
\end{theorem}

The main method we prove here is to construct appropriate discrete upper and lower solutions $\{u_n\}$ and $\{w_n\}$  such that they all have the same long-term asymptotic decay rate $O(t_{n}^{-(\alpha+\beta)/\gamma})$, and then obtain the desired result using the discrete comparison principle.

\begin{proof} 

\textbf{I. Construction of the discrete subsolution} Set $g_{n}:=t_{n}^{\alpha+\beta}/\Gamma(1+\alpha+\beta)$ for $n\geq 0$. Define the sequence
\begin{equation}\label{vn}
    u_{n}=\left\{\begin{aligned}
                  &y_{0}-\mu g_{n}~~for~n=0,1,\dots,n_{1},\\
                  &C_{7}t_{n}^{-\frac{\alpha+\beta}{\gamma}}~~for~n\geq n_{1}+1,
                 \end{aligned}\right.
\end{equation}
where $\mu>0$, $C_{7}>0$ and $n_{1}\in\mathbb{N}$ are three parameters to be determined. 
Below, we will make appropriate choices for these three parameters to ensure that:

 (i) the sequence $\{u_{n}\}$ is monotonically decreasing and positive;
  
 (ii) satisfies that $D_{\tau}^{\alpha}(u_{n})+\nu t_{n}^{\beta}u_{n}^{\gamma}\leq 0$.\\

We divide the proof into several steps.

\emph{Step 1. $\{u_{n}\}$~is monotonically decreasing.}
From \eqref{vn} we know that $g_{n}$ is monotonically increasing with respect to $n$ for $1\leq n\leq n_{1}$ and  $\mu >0$, $C_{7}>0$. 
Therefore, in order for $\{u_n\}$ to be positive and monotonically decreasing if and only if $u_{n_{1}}\geq u_{n_{1}+1}$.
That is 
\begin{equation}\label{vn1}
   y_{0}-\frac{\mu}{\Gamma(1+\alpha+\beta)}t_{n_{1}}^{\alpha+\beta}\geq C_{7}t_{n_{1}+1}^{-\frac{\alpha+\beta}{\gamma}}.
\end{equation}

\emph{Step 2. $\{u_{n}\}$~is a discrete subsolution for $1\leq n\leq n_{1}$.}
For $1\leq n\leq n_{1}$, it follows from \eqref{dis-caputo} that
\begin{align*}
     D_{\tau}^{\alpha}(u_{n})&=\frac{1}{\tau^{\alpha}}\left(\delta_{n}u_{0}+\sum_{k=1}^{n}\omega_{n-k}u_{k}\right)
     =\frac{1}{\tau^{\alpha}}\left(\left(\sum_{k=1}^{n}\omega_{n-k}+\delta_{n}\right) u_{0}-\mu\sum_{k=1}^{n}\omega_{n-k}g_{k}\right)
     =-\frac{\mu}{\tau^{\alpha}}\sum_{k=1}^{n}\omega_{n-k}g_{k}.
\end{align*}
By the properties \eqref{coe-pro1}, we have the estimate
\begin{equation*}
    \sum_{k=1}^{n}\omega_{n-k}g_{k}=\omega_{0}g_{n}+\sum_{k=1}^{n-1}\omega_{k}g_{n-k}=\sum_{k=1}^{n-1}\omega_{k}(g_{n-k}-g_{n})+\left(-\sum_{k=n}^{\infty}\omega_{k}g_{n}\right)\geq-\sum_{k=n}^{\infty}\omega_{k}g_{n}\geq \frac{C_{3}}{n^{\alpha}}g_{n},
\end{equation*}
which yields that $D_{\tau}^{\alpha}(u_{n})\leq -\frac{\mu}{\tau^{\alpha}}\frac{C_{3}}{n^{\alpha}}g_{n}=-\frac{\mu C_{3}}{\Gamma(1+\alpha+\beta)}t_{n}^{\beta}$.
Therefore, in order to make $\{u_{n}\}$ a discrete lower solution, we only need to select parameters such that
\[
\frac{\mu C_{3}}{\Gamma(1+\alpha+\beta)}t_{n}^{\beta}\geq \nu t_{n}^{\beta}(y_{0}-\mu g_{n})^{\gamma}\geq \nu t_{n}^{\beta}y_{0}^{\gamma}.
\]
Hence,  we can take $\mu$ such that
\begin{equation}\label{vn2}
\mu \geq \nu y_{0}^{\gamma}\frac{\Gamma(1+\alpha+\beta)}{C_{3}}.
\end{equation}

\emph{Step 3. $\{u_{n}\}$~is a discrete subsolution for $ n\geq n_{1}+1$.}
For $n\geq n_{1}+1$, we have 
\begin{align*}
    D_{\tau}^{\alpha}(u_{n})=\frac{1}{\tau^{\alpha}}\left(\delta_{n}u_{0}+\sum_{k=1}^{n}\omega_{n-k}u_{k}\right)
    =\frac{1}{\tau^{\alpha}}\underset{:=I_{1}}{\underbrace{\left(\sum_{k=1}^{n_{0}}\omega_{n-k}u_{k}+\delta_{n}u_{0}\right)}}+\frac{1}{\tau^{\alpha}}\underset{:=I_{2}}{\underbrace{\sum_{k=n_{0}+1}^{n}\omega_{n-k}u_{k}}}.
\end{align*}
Note that $I_{1}$ can be rewritten as
$I_{1}=\left(\sum_{k=1}^{n_{0}}\omega_{n-k}+\delta_{n}\right)u_{0}-\mu\sum_{k=1}^{n_{0}}\omega_{n-k}g_{k}=-\sum_{k=0}^{n-n_{0}-1}\omega_{k}u_{0}-\mu\sum_{k=1}^{n_{0}}\omega_{n-k}g_{k}.$
Hence, recalling that $\omega_{k}<0$ for $k\geq1$, we have
\begin{align*}
    I_{1}+I_{2}=\omega_{0}(u_{n}-u_{0})+\sum_{k=1}^{n-n_{0}-1}|\omega_{k}|(u_{0}-u_{n-k})+\mu\sum_{k=1}^{n_{0}}|\omega_{n-k}|g_{k}.
\end{align*}
In view of $u_{0}=u_{0}-u_{n_{1}}+u_{n_{1}}=u_{n_{1}}+\mu g_{n_{1}}$ and $\omega_{0}=-\sum_{k=1}^{\infty}\omega_{k}=\sum_{k=1}^{\infty}|\omega_{k}|>0$, we get that
\begin{align*}
    I_{1}+I_{2}=\sum_{k=n-n_{1}}^{\infty }|\omega_{k}| (u_{n}-u_{n_{1}})+\sum_{k=1}^{n-n_{1}-1}|\omega_{k}| (u_{n}-u_{n-k})+\mu\sum_{k=1}^{n_{1}}|\omega_{n-k}|g_{k}+\mu g_{n_{1}}\left(\sum_{k=1}^{n-n_{1}-1}|\omega_{k}|-\omega_{0}\right).
\end{align*}
Since $\{u_{n}\}$ is monotonically decreasing, so
\begin{align*}
    I_{1}+I_{2}&\leq\mu\sum_{k=1}^{n_{1}}|\omega_{n-k}|g_{k}+\mu g_{n_{1}}\sum_{k=n-n_{1}}^{\infty}\omega_{k}\\
    &\leq \mu g_{n_{1}}\sum_{k=1}^{n_{1}}|\omega_{n-k}|+\mu g_{n_{1}}\sum_{k=n-n_{1}}^{\infty}\omega_{k}=\mu g_{n_{1}}\sum_{k=n}^{\infty}\omega_{k}\\
    &\leq -\mu g_{n_{1}}\frac{C_{3}}{n^{\alpha}}=-\nu y_{0}^{\gamma} t_{n_{1}}^{\alpha+\beta}\frac{1}{n^{\alpha}}.
\end{align*}
Thus $D_{\tau}^{\alpha} (u_{n})=\frac{1}{\tau^{\alpha}}(I_{1}+I_{2})\leq -\nu y_{0}^{\gamma}t_{n_{1}}^{\alpha+\beta}t_{n}^{-\alpha}.$
We just need to select parameters such that $D_{\tau}^{\alpha}(u_{n})\leq -\nu y_{0}^{\gamma}t_{n_{1}}^{\alpha+\beta}t_{n}^{-\alpha} \leq -\nu t_{n}^{\beta}u_{n}^{\gamma} (:=-\nu C_{7}^{\gamma}t_{n}^{-\alpha}).$
That is
\begin{equation}\label{vn3}
    y_{0}^{\gamma}t_{n_{1}}^{\alpha+\beta}\geq C_{7}^{\gamma}.
\end{equation}

\emph{Step 4. Parameters selection relying only on data.}
We now choose appropriate parameters to ensure that all three inequalities \eqref{vn1},  \eqref{vn2} and \eqref{vn3} hold simultaneously. We can choose 
 \begin{equation}\label{un-coe}
 \mu:=\nu y_{0}^{\gamma}\frac{\Gamma(1+\alpha+\beta)}{C_{3}}, \quad
t_{n_{1}}:=\max\left\{t_{n}:t_{n}^{\alpha+\beta}\leq\frac{C_{3}y_{0}^{1-\gamma}}{2\nu}\right\},
\quad C_{7}:=\frac{y_{0}t_{n_{1}}^{\frac{\alpha+\beta}{\gamma}}}{2}.
\end{equation}
According to the above definition, we can easily see that inequality \eqref{vn2} and \eqref{vn3} hold. Let's verify that \eqref{vn1} is also valid.
In fact, we have
\[
u_{n_{1}}=y_{0}-\mu g_{n_{1}}=y_{0}-\nu y_{0}^{\gamma}\frac{t_{n_{1}}^{\alpha+\beta}}{C_{3}} \geq y_{0}-\nu y_{0}^{\gamma}\frac{1}{C_{3}}\frac{C_{3}u_{0}^{1-\gamma}}{2\nu}\geq  \frac{y_{0}}{2}, \quad 
u_{n_{1}+1}=C_{7}t_{n_{1}+1}^{-\frac{\alpha+\beta}{\gamma}}=\frac{y_{0}}{2}\left(\frac{t_{n_{1}}}{t_{n_{1}+1}}\right)^{\frac{\alpha+\beta}{\gamma}}\leq\frac{y_{0}}{2}.
\]
Then $u_{n_{1}}\geq u_{n_{1}+1}$. 
Those analysis shows that  $\{u_{n}\}$ is a subsolution to $\{y_{n}\}$ with decay rate $O\left( t_{n}^{-\frac{\alpha+\beta}{\gamma}} \right)$.

\vskip 0.2cm
\textbf{II. Construction of the discrete supersolution.}  Define the function
\begin{equation}\label{wn}
    w(t)=\left\{\begin{aligned}
        &y_{0}~~for~0\leq t\leq t_{n_{2}},\\
        &C_{8}t^{-\frac{\alpha+\beta}{\gamma}}~~for~t_{n_{2}}<t<\infty.
    \end{aligned}\right.
\end{equation}
Then take $w_{n}=w(t_n)$ for $n\geq 0$. The parameters $C_{8}$  and $t_{n_{2}}$ are two parameters to be determined. 
Below, we will make appropriate choices for these two parameters to ensure that:

 (i) the sequence $\{w_{n}\}$ is monotonically decreasing and positive;
  
 (ii) satisfies that $D_{\tau}^{\alpha}(w_{n})+\nu t_{n}^{\beta}w_{n}^{\gamma}\geq 0$.\\

We also divide the proof into several steps.

\emph{Step 1. $\{w_{n}\}$~is monotonically decreasing.}
Obviously when $C_{8}>0$, $w(t)$ is monotonically decreasing on $(t_{n_{2}},\infty)$. In order for $w(t)$ is monotonic for all $t\geq 0$, we just need
$w_{n_2+1}\leq w_{n_2}$. That is $C_{8}t_{n_2+1}^{-\frac{\alpha+\beta}{\gamma}}\leq y_{0}$.
Now we take $C_{8}:=y_{0}t_{n_{2}}^{\frac{\alpha+\beta}{\gamma}}$ to make the required inequality holds. 

\emph{Step 2. $\{w_{n}\}$~is a discrete supersolution for $1\le n\le n_2$.}
From \eqref{wn}, we have
 \begin{equation}\label{Dwn1}
     D_{\tau}^{\alpha}(w_{n})=0>-\nu t_{n}^{\beta}w_{n}^{\gamma} ~~for~n=1,2,\dots, n_{2}.
 \end{equation}

\emph{Step 3. $\{w_{n}\}$~is a discrete supersolution for $n_2\le n\le 2n_2$.}
Considering separately the cases $n_{2}<n\leq 2n_{2}$ and $n>2n_{2}$. 
Note that $\delta_{k}=-\sum_{\ell=0}^{k-1}\omega_{\ell}$ for $k\ge 1$ and put $\delta_{0}=0$. Hence $\omega_{k}=\delta_{k}-\delta_{k+1}$.
With this formula, the discretisation \eqref{dis-caputo} can be rewritten as
\begin{equation}\label{Dh2}
D_{\tau}^{\alpha}(w_{n})=\frac{1}{\tau^{\alpha}}\sum_{k=1}^{n}\delta_{k}(w_{n-k}-w_{n-k-1})=\frac{1}{\tau^{\alpha}}\sum_{k=1}^{n-n_{2}}\delta_{k}(w_{n-k}-w_{n-k-1}),
\end{equation}
where we use the fact $w_{n-k}-w_{n-k-1}=0$ for $k\geq n-n_{2}+1$.
Note that $w'(t)<0$ and $w''(t)>0$ for $t\ge t_{n_{2}}$, which leads to $w_{n-k}-w_{n-k-1}\le w_{n_2}-w_{n_2+1}$ for $1\le k\le n-n_{2}$.
Now it follows from \eqref{Dh2} that
\begin{equation*}
    D_{\tau}^{\alpha}(w_{n})\geq \frac{1}{\tau^{\alpha}}\sum_{k=1}^{n-n_{2}}\delta_{k}(w_{n_{2}}-w_{n_{2}+1})=-\frac{1}{\tau^{\alpha}}(w_{n_{2}}-w_{n_{2}+1})\sum_{k=1}^{n-n_{2}}|\delta_{k}|.
\end{equation*}
By \eqref{coe-pro2} and $n_{2}<n\leq 2n_{2}$, one has
$ \sum_{k=1}^{n-n_{2}}|\delta_{k}|\leq \sum_{k=1}^{n-n_{2}}C_{4}k^{-\alpha}\leq C_{4}\int_{s=0}^{n-n_{2}}s^{-\alpha}ds=\frac{C_{4}(n-n_{2})^{1-\alpha}}{1-\alpha}\leq \frac{C_{4}n_{2}^{1-\alpha}}{1-\alpha},
$
so
\begin{equation}\label{wtn1}
    D_{\tau}^{\alpha}(w_{n})\geq-\frac{1}{\tau^{\alpha}}(w_{n_{2}}-w_{n_{2}+1})\frac{C_{4}n_{2}^{1-\alpha}}{1-\alpha}\geq -\frac{C_{4}\tau^{1-\alpha}n_{2}^{1-\alpha}}{1-\alpha}\frac{C_{8}(\alpha+\beta)}{\gamma}t_{n_{2}}^{-\frac{\alpha+\beta}{\gamma}-1},
\end{equation}
where $w'<0$ and $w''>0$ on $(t_{n_{2}},t_{n_{2}+1})$ imply that
$0<w_{n_{2}}-w_{n_{2}+1}\leq -\tau w'(t_{n_{2}})=\tau\frac{C_{8}(\alpha+\beta)}{\gamma}t_{n_{2}}^{-\frac{\alpha+\beta}{\gamma}-1}$. 
By the definition of $C_{8}$ and $n\leq 2n_{2}$, there is
\begin{equation}\label{wtn2}
    D_{\tau}^{\alpha}(w_{n})\geq -\frac{\nu C_{4}y_{0}(\alpha+\beta)}{(1-\alpha)\gamma}t_{n_{2}}^{-\alpha}\geq -\frac{\nu C_{4}2^{\alpha}y_{0}^{1-\gamma}(\alpha+\beta)}{(1-\alpha)\gamma}y_{0}^{\gamma}t_{n}^{-\alpha}.
\end{equation}
In order for $\{w_{n}\}$ is a discrete supsolution, we only need 
\begin{equation}\label{wnsuper}
    D_{\tau}^{\alpha}(w_{n})\geq
    -\frac{\nu C_{4}2^{\alpha}y_{0}^{1-\gamma}(\alpha+\beta)}{(1-\alpha)\gamma}y_{0}^{\gamma}t_{n}^{-\alpha}
    \geq -\nu t_{n}^{\beta}w_{n}^{\gamma},  \text{ where }\nu t_{n}^{\beta}w_{n}^{\gamma}=\nu C_{8}^{\gamma}t_{n}^{-(\alpha+\beta)}t_{n}^{\beta}=\nu y_{0}^{\gamma}t_{n_{2}}^{\alpha+\beta}t_{n}^{-\alpha}.
\end{equation}
Therefor, we can take  $t_{n_{2}}$ such that
\begin{equation}\label{tn11}
     t_{n_{2}}^{\alpha+\beta}\geq \frac{y_{0}^{1-\gamma}}{\nu}\frac{(\alpha +\beta)2^{\alpha}C_{4}}{\gamma(1-\alpha)}.
\end{equation}

\emph{Step 4. $\{w_{n}\}$~is a discrete supersolution for $ n\ge 2n_2+1$.}
 Let $n':=\lceil \frac{n}{2}\rceil$, so $n'>n_{2}$ and $n'\leq \frac{n+1}{2}$. From \eqref{Dh2}, one has
\begin{align}\label{Dw2}
    D_{\tau}^{\alpha}(w_{n})\geq-\frac{1}{\tau^{\alpha}}\underset{:=I_{3}}{\underbrace{\sum_{k=1}^{n-n'}C_{4}k^{-\alpha}(w_{n-k}-w_{n-k-1})}}-\frac{1}{\tau^{\alpha}}\underset{:=I_{4}}{\underbrace{\sum_{k=n-n'+1}^{n-n_{1}}C_{4}k^{-\alpha}(w_{n-k}-w_{n-k-1})}}.
\end{align}
Using $n'>n_{2}$ and the same method as that in the case of $n<2n_{2}$, one has
\begin{equation}\label{I3}
    \begin{aligned}
        I_{3}\leq C_{4}(w_{n'}-w_{n'+1})\sum_{k=1}^{n-n'}k^{-\alpha}\leq C_{4}\tau|w'(t_{n'})|\int_{s=0}^{n-n'}s^{-\alpha}ds=C_{4}\tau\frac{C_{8}(\alpha+\beta)}{\gamma}t_{n'}^{-\frac{\alpha+\beta}{\gamma}-1}\frac{(n-n')^{1-\alpha}}{1-\alpha}.
    \end{aligned}
\end{equation}
Note that $n'\geq\frac{n}{2}$ implies that $t_{n'}^{-\frac{\alpha+\beta}{\gamma}-1}\leq 2^{\frac{\alpha+\beta}{\gamma}+1}t_{n}^{-\frac{\alpha+\beta}{\gamma}-1}$ and $(n-n')^{1-\alpha}\leq 2^{1-\alpha}n^{1-\alpha}$. Now by the definition of $C_{8}$, we have 
\begin{align*}
    I_{3}\leq C_{4}\tau\frac{(\alpha+\beta)}{\gamma}y_{0}t_{n_{2}}^{\frac{\alpha+\beta}{\gamma}}2^{\frac{\alpha+\beta}{\gamma}+1}t_{n}^{-\frac{\alpha+\beta}{\gamma}-1}\frac{n^{1-\alpha}}{1-\alpha}2^{\alpha-1}\leq C_{4}\tau\frac{(\alpha+\beta)}{\gamma}y_{0}2^{\frac{\alpha+\beta}{\gamma}+1}t_{n}^{-1}\frac{n^{1-\alpha}}{1-\alpha}2^{\alpha-1}.
\end{align*}

On the other hand, $I_{4}$ can be bounded by
$ I_{4}\leq C_{4}(n-n'+1)^{-\alpha}\sum_{k=n-n'+1}^{n-n_{2}}(w_{n-k}-w_{n-k-1})=C_{4}(n-n'+1)^{-\alpha}(w_{n_{2}}-w_{n'})$
Since $n'\leq \frac{n+1}{2}$ and the definitions of $w$, then
\begin{equation}\label{I4}
    I_{4}\leq C_{4}\left(\frac{n+1}{2}\right)^{-\alpha}(w_{n_{2}}-w_{n'})\leq C_{4} 2^{\alpha}n^{-\alpha}y_{0}.
\end{equation}

Combining \eqref{Dw2}, \eqref{I3} and \eqref{I4} give
\begin{align}\label{wnC9}
    D_{\tau}^{\alpha}(w_{n})\geq -t_{n}^{-\alpha}C_{4} 2^{\alpha}y_{0}\left[\frac{(\alpha+\beta)}{\gamma(1-\alpha)}2^{\frac{\alpha+\beta}{\gamma}}+1 \right].
\end{align}
In order for $\{w_{n}\}$ is a discrete supsolution, we only need 
\begin{equation}\label{wnsup}
    D_{\tau}^{\alpha}(w_{n})\geq
    -t_{n}^{-\alpha}C_{4} 2^{\alpha}y_{0}\left[\frac{(\alpha+\beta)}{\gamma(1-\alpha)}2^{\frac{\alpha+\beta}{\gamma}}+1 \right]
    \geq -\nu t_{n}^{\beta}w_{n}^{\gamma},  \text{ where }\nu t_{n}^{\beta}w_{n}^{\gamma}=\nu C_{8}^{\gamma}t_{n}^{-(\alpha+\beta)}t_{n}^{\beta}=\nu y_{0}^{\gamma}t_{n_{2}}^{\alpha+\beta}t_{n}^{-\alpha},
\end{equation}
This leads to the condition 
\begin{equation}\label{tn12}
    t_{n_{2}}^{\alpha+\beta}\geq \frac{y_{0}^{1-\gamma}}{\nu}2^{\alpha}C_{4}\left[\frac{(\alpha+\beta) 2^{\frac{\alpha+\beta}{\gamma}}}{\gamma(1-\alpha)}+1\right].
\end{equation}

From \eqref{tn11} and \eqref{tn12}, we can take the parameter $t_{n_{2}}$ as 
\begin{equation}\label{tn1}
    t_{n_{2}}:=\min \left\{t_{n}:t_{n}^{\alpha+\beta}\geq \frac{y_{0}^{1-\gamma}}{\nu}2^{\alpha}C_{4} \cdot \max\left\{\frac{(\alpha +\beta)}{\gamma(1-\alpha)}, \frac{(\alpha+\beta) 2^{(\alpha+\beta)/\gamma}}{\gamma(1-\alpha)}+1\right\}\right\},
\end{equation}
which can make $\{w_{n}\}$ is a supersolution for $\{y_{n}\}$.
\end{proof}

\subsection{Numerical Mittag-Leffler stability for non-homogeneous equations} \label{sec:33}

Our goal here is study the decay estimates of numerical solutions to the non-homogeneous fractional ODE: $D_{t}^{\alpha}\varphi(t)+\nu t^{\beta}\varphi^{\gamma}(t)=f(t)$ with $\varphi(0)=\varphi_{0}>0$, where all the parameters and $f$ satisfy the same conditions as Theorem \ref{t-nonli-f}.
Hence, Theorem \ref{dis-nonli-f} below is the discrete version of Theorem \ref{t-nonli-f}.

\begin{theorem}\label{dis-nonli-f}
    Consider the time $\mathcal{CM}$-preserving scheme for the fractional ODE given by
    \begin{equation}\label{dis-non-ode-f}
        D_{\tau}^{\alpha}(\varphi_{n})+\nu t_{n}^{\beta}\varphi_{n}^{\gamma}=f_{n}~for~n\geq 1,
    \end{equation} 
    On the uniform mesh $\{t_{n}=n\tau\}_{n=0}^{\infty}$, the numerical solution $\{\varphi_{n}\}$ satisfies
    \begin{equation}\label{un-decay-f}
        \frac{\breve{C}_{5}}{1+t_{n}^{(\alpha+\beta)/\gamma}}\leq \varphi_{n}\leq \frac{\breve{C}_{6}}{1+t_{n}^{(\alpha+\beta)/\gamma}}~~for~n=0,1,2,\dots,
    \end{equation}
    where the positive constants $\breve{C}_{5}$, $\breve{C}_{6}$ are independent of $n$ and $\tau$.
\end{theorem}

The proof strategy is somewhat similar to the Theorem \ref{t-nonli-f}. Under the assumption of source function $f$, we can control the additional terms brought by the source function. Of course, compared to the case of homogeneous equations, this requires more precise estimation and better discrete upper and lower solutions.
Due to the detailed proof of Theorem \ref{uni-decay} on homogeneous equations, we will only provide an overview of the proof of this theorem here.

\begin{proof}
    From $f$ satisfying the decay estimate $|f_{n}|\leq t_{n}^{\beta}/(t_{n}+1)^{\alpha+\beta}$, we know that 
    \begin{equation*}
        -\frac{Kt_{n}^{\beta}}{(t_{n}+1)^{\alpha+\beta}}\leq D_{\tau}^{\alpha}(\varphi_{n})+\nu t_{n}^{\beta}\varphi_{n}^{\gamma}\leq\frac{Kt_{n}^{\beta}}{(t_{n}+1)^{\alpha+\beta}}.
    \end{equation*}
Our current strategy is to construct appropriate discrete sub solution $\{u_{n}\}$ and sup solution $\{w_{n}\}$ such that
    \begin{equation}\label{subsup}
        D_{\tau}^{\alpha}(u_{n})+\nu t_{n}^{\beta}u_{n}^{\gamma}\leq  -\frac{Kt_{n}^{\beta}}{(t_{n}+1)^{\alpha+\beta}}\le D_{\tau}^{\alpha}(\varphi_{n})+\nu t_{n}^{\beta}\varphi_{n}^{\gamma}\leq \frac{Kt_{n}^{\beta}}{(t_{n}+1)^{\alpha+\beta}} \le D_{\tau}^{\alpha}(w_{n})+\nu t_{n}^{\beta}w_{n}^{\gamma}.
    \end{equation}

\textbf{I. Discrete subsolution.}  Define
    \begin{equation*}
\theta:=\frac{2\varphi_{0}^{\gamma}\nu \Gamma(1+\alpha+\beta)}{C_{3}},~C_{10}=2^{-1}\varphi_{0}t_{n_{3}}^{(\alpha+\beta)/\gamma},~ g_{n}=\frac{t_{n}^{\alpha+\beta}}{\Gamma(1+\alpha+\beta)},
    \end{equation*}
where $t_{n_{3}}:=\min \left\{t_n: t_n\ge \left(\frac{C_{3}\varphi_{0}^{1-\gamma}}{4\nu }\right)^{1/(\alpha+\beta)} \right\}$ and $\varphi_{0}^{\gamma}:=\max\left\{K/\nu,\left[2K/C_{3}(1-2^{-(\gamma+1)})\right]^{1/\gamma}\right\}$.
    Then define the monotonically decreasing sequence
    \begin{equation}\label{vn-f}
        u_{n}=\left\{\begin{aligned}
                  &\varphi_{0}-\theta g_{n}~~for~n=0,1,\dots,n_{3},\\
                  &C_{10}t_{n}^{-(\alpha+\beta)/\gamma}~~for~n\geq n_{3}+1.
                 \end{aligned}\right.
    \end{equation}
    
    For $n\le n_{3}$,  we have
    \begin{equation}\label{sub-1}
        \begin{aligned}
            D_{\tau}^{\alpha}(u_{n})+\nu t_{n}^{\beta}u_{n}^{\gamma}&\leq -t_{n}^{\beta}\left(\frac{\theta C_{3}}{\Gamma(1+\alpha+\beta)}-\nu(\varphi_{0}-\theta g_{n})^{\gamma}\right)\\
            &\leq  -t_{n}^{\beta}\left(\frac{\theta C_{3}}{\Gamma(1+\alpha+\beta)}-\nu\varphi_{0}^{\gamma}\right)
            \leq -Kt_{n}^{\beta}\leq -\frac{Kt_{n}^{\beta}}{(t_{n}+1)^{\alpha+\beta}}.
        \end{aligned}  
    \end{equation}

  For $n>n_{3}$, by the definition of $t_{n_{3}}$, $C_{10}$ and $\theta$ we have
    \begin{equation}\label{sub-2}
        \begin{aligned}
            D_{\tau}^{\alpha}(u_{n})+\nu t_{n}^{\beta}u_{n}^{\gamma}&\leq -\frac{\theta C_{3}}{\Gamma(1+\alpha+\beta)}t_{n_{3}}^{\alpha+\beta}t_{n}^{\alpha}+\nu t_{n}^{\beta}C_{10}^{\gamma}t_{n}^{-(\alpha+\beta)}\\
            &=-t_{n}^{\beta}t_{n}^{-(\alpha+\beta)}\left(\frac{\theta C_{3}}{\Gamma(1+\alpha+\beta)}t_{n_{3}}^{\alpha+\beta}-\nu C_{10}^{\gamma}\right)
            \leq-\frac{t_{n}^{\beta}}{(t_{n}+1)^{\alpha+\beta}}\frac{C_{3}\varphi_{0}}{2}(1-2^{-(\gamma+1)}).
        \end{aligned}  
    \end{equation}
Noting that $\varphi_{0}^{\gamma}\geq\left[2K/C_{3}(1-2^{-(\gamma+1)})\right]^{1/\gamma}$, we get that
$D_{\tau}^{\alpha}(u_{n})+\nu t_{n}^{\beta}u_{n}^{\gamma}\leq -\frac{Kt_{n}^{\beta}}{(t_{n}+1)^{\alpha+\beta}}.$

\vskip0.2cm
    \textbf{II. Discrete supersolution.} Define the function
    \begin{equation}\label{wn-f}
    w(t)=\left\{\begin{aligned}
                 &\varphi_{0}~~for~0\leq t\leq t_{n_{4}},\\
                 &\varphi_{0}t_{n_{4}}^{\frac{\alpha+\beta}{\gamma}}t^{-\frac{\alpha+\beta}{\gamma}}~~for~t_{n_{4}}<t<\infty,
              \end{aligned}\right.
    \end{equation}
    where $t_{n_{4}}:=\min\left\{t_n: t_n\geq \left[\frac{C_{4}\varphi_{0}^{1-\gamma}2^{\alpha}}{\nu}\left(\frac{(\alpha+\beta)}{\gamma(1-\alpha)}2^{\frac{\alpha+\beta}{\gamma}}+1\right)+1\right]^{1/(\alpha+\beta)}\right\}$. Then we define the sequence $w_{n}=w(t_n)$ for $n\geq 0$.
      
    For $1\leq n \leq n_{4}$, we have $\varphi_{0}^{\gamma}\geq K/\nu$ and
    $ 
    D_{\tau}^{\alpha}(w_{n})+\nu t_{n}^{\beta}w_{n}^{\gamma}= \nu t_{n}^{\beta}\varphi_{0}^{\gamma}\geq Kt_{n}^{\beta}\geq \frac{Kt_{n}^{\beta}}{(t_{n}+1)^{\alpha+\beta}}.
    $

    Next, suppose that $n_{4}<n<2n_{4}$, one has 
    \begin{equation}\label{sup-2}
    \begin{aligned}
        D_{\tau}^{\alpha}(w_{n})+\nu t_{n}^{\beta}w_{n}^{\gamma}&\geq \frac{-\nu C_{4}\varphi_{0}^{1-\gamma}(\alpha+\beta)}{(1-\alpha)\gamma}\varphi_{0}^{\gamma}t_{n}^{-\alpha}+\nu t_{n}^{\beta}\varphi_{0}^{\gamma}t_{n_{4}}^{\alpha+\beta}t_{n}^{-(\alpha+\beta)}\\
        &\geq \nu t_{n}^{\beta}\varphi_{0}^{\gamma}t_{n}^{-(\alpha+\beta)}\left(\frac{- C_{4}\varphi_{0}^{1-\gamma}(\alpha+\beta)}{(1-\alpha)\gamma}\varphi_{0}^{\gamma}+t_{n_{4}}^{\alpha+\beta}\right).
    \end{aligned}
    \end{equation}
    By the definition of $t_{n_{4}}$, we have $\frac{- C_{4}2^{\alpha}\varphi_{0}^{1-\gamma}(\alpha+\beta)}{(1-\alpha)\gamma}+t_{n_{4}}^{\alpha+\beta}\geq1$. Noting that $\varphi_{0}^{\gamma}\geq K/\nu$, we obtain $D_{\tau}^{\alpha}(w_{n})+\nu t_{n}^{\beta}w_{n}^{\gamma}\geq\frac{Kt_{n}^{\beta}}{(t_{n}+1)^{\alpha+\beta}}$.

    Finally, suppose that $n\geq 2n_{4}$.  We have
    \begin{equation}\label{sup-3}
    \begin{aligned}
        D_{\tau}^{\alpha}(w_{n})+\nu t_{n}^{\beta}w_{n}^{\gamma}&\geq -t_{n}^{-\alpha}\varphi_{0}C_{4}2^{\alpha}\left(\frac{(\alpha+\beta)}{\gamma(1-\alpha)}2^{\frac{\alpha+\beta}{\gamma}}+1\right)+\nu t_{n}^{\beta}\varphi_{0}^{\gamma}t_{n_{4}}^{\alpha+\beta}t_{n}^{-(\alpha+\beta)}\\
        &\geq \nu t_{n}^{\beta}\varphi_{0}^{\gamma}t_{n}^{-(\alpha+\beta)}\left[-\frac{C_{4}\varphi_{0}^{1-\gamma}2^{\alpha}}{\nu}\left(\frac{(\alpha+\beta)}{\gamma(1-\alpha)}2^{\frac{\alpha+\beta}{\gamma}}+1\right)+t_{n_{4}}^{\alpha+\beta}\right].
    \end{aligned}
    \end{equation}
It follows from the definition of $t_{n_{4}}$ that $-\frac{C_{4}\varphi_{0}^{1-\gamma}2^{\alpha}} {\nu}\left(\frac{(\alpha+\beta)}{\gamma(1-\alpha)}2^{\frac{\alpha+\beta}{\gamma}}+1\right)+t_{n_{4}}^{\alpha+\beta}\ge 1$. Hence, we get that 
$D_{\tau}^{\alpha}(w_{n})+\nu t_{n}^{\beta}w_{n}^{\gamma}\geq \frac{Kt_{n}^{\beta}}{(t_{n}+1)^{\alpha+\beta}}$ by the fact $\varphi_{0}^{\gamma}\geq K/\nu$.

   In summary, the above analysis indicates that the discrete sub solution $\{u_n\}$ and sup solution $\{w_n\}$ satisfy the inequality \eqref{subsup}. Hence, it follows from Theorem \ref{uni-decay} that $u_n\le \varphi_{n} \le w_n$, which leads to the required results by the construction of $\{u_n\}$ and  $\{w_n\}$.
  
\end{proof}

\subsection{Numerical Mittag-Leffler stability for time fractional PDEs} \label{sec:34}

Consider the time semi-discretization of the initial boundary value problem \eqref{pde_f} by $\mathcal{CM}$-preserving schemes 
\begin{equation}\label{Un}
    D_{\tau}^{\alpha}U_{n}+ \nu t_{n}^{\beta}\mathcal{N}(U_{n}(x))=f_{n}(x)~~in~\Omega,~U_{n}=0~on~\partial\Omega,~U_{0}=u_{0},
\end{equation}
where $U_{n}=U_{n}(x)$ is the approximation of $u(t_{n},\cdot)$ at $t_{n}=n\tau$. 
Now we hope to apply the estimation of fractional ODE obtained above and combine appropriate energy methods to establish the long-term optimal decay rate of the numerical solution $\{\|U_{n}\|_{L^{s}(\Omega)}\}$, that is, to establish a discrete version of Corollary \ref{c-decay-f-N}.

The following $L^{s}$-norm inequality of $\mathcal{CM}$-preserving schemes is crucial. It can be seen as a discrete version of the $L^{s}$-norm inequality \eqref{dt-inq}, which once again reveals the good structural preservation characteristics of $\mathcal{CM}$-preserving schemes.

\begin{lemma}\cite{CM-preserving}\label{num-dut}
    Assume that for each $n\geq 0$ the discrete solution $U_{n}\in L^{s}(\Omega)$ for some $s\in (1,\infty)$. Consider $\mathcal{CM}$-preserving scheme of $D_{\tau}^{\alpha}(\cdot)$ defined by \eqref{coe-pro1}. Then for any $n\geq1$, there holds
        \begin{equation}\label{disdt-inq}
        \|U_{n}\|_{L^{s}(\Omega)}^{s-1}D_{\tau}^{\alpha}\|U_{n}\|_{L^{s}(\Omega)}\leq \int_{\Omega}(U_{n}(x))^{s-1}D_{\tau}^{\alpha}(U_{n})(x)dx.
    \end{equation}
\end{lemma}

\begin{theorem}\label{t42}
    Let $U_{n}$ be the nonnegative semi-discretization weak solution of \eqref{Un} for each $n\geq1$.
    The parameters and $f$ satisfy the exact same assumptions as in Corollary \ref{c-decay-f-N}.
    Then there exists $\breve{C}>0$ such that
 \[
 \|U_{n}\|_{L^{s}(\Omega)}\leq \frac{\breve{C}}{1+t_{n}^{\frac{\alpha+\beta}{\gamma}}}~for~n=1,2,\dots.
 \]
\end{theorem}
\begin{proof}
 As in Corollary \ref{c-decay-f-N}, We multiply the equation \eqref{Un} both sides by $U_{n}^{s-1}(x)$ and integrate over $\Omega$ to get
    \begin{equation*}
        \int_{\Omega}U_{n}^{s-1}(x)D_{\tau}^{\alpha}U_{n}(x)dx+\int_{\Omega}\nu t_{n}^{\beta}\mathcal{N}(U_{n}(x))U_{n}^{s-1}(x)dx = \int_{\Omega}f_{n}(x)U_{n}^{s-1}(x)dx.
    \end{equation*}
The structural assumption \eqref{str-ass1} and the $L^{s}$-norm inequality \eqref{disdt-inq} imply that
    \begin{equation*}
        \|U_{n}\|_{L^{s}(\Omega)}^{s-1}D_{\tau}^{\alpha}\|U_{n}\|_{L^{s}(\Omega)}+\|U_{n}(x)\|_{L^{s}(\Omega)}^{s-1+\gamma}\leq  \int_{\Omega}f_{n}(x)U_{n}^{s-1}(x)dx.
    \end{equation*}
    By H$\mathrm{\Ddot{o}}$lder inequality and \eqref{f_esti}, we have
    \begin{equation*}
        \int_{\Omega}f_{n}(x)U_{n}^{s-1}(x)dx\leq \|f_{n}\|_{L^{s}(\Omega)}\|U_{n}(x)\|_{L^{s}(\Omega)}^{s-1}\leq \frac{Kt_{n}^{\beta}}{(t_{n}+1)^{\alpha+\beta}}\|U_{n}(x)\|_{L^{s}(\Omega)}^{s-1},
    \end{equation*}
    then $D_{\tau}^{\alpha}\|U_{n}\|_{L^{s}(\Omega)}+C_{N}\nu t_{n}^{\beta}\|U_{n}\|^{\gamma}_{L^{s}(\Omega)}\leq \frac{Kt_{n}^{\beta}}{(t_{n}+1)^{\alpha+\beta}}$.
   Next, let $\varphi_{n};=\|U_{n}\|_{L^{s}(\Omega)}$ and apply Theorem \ref{dis-nonli-f}, we can obtain the required results.
    \end{proof}

\section{Numerical experiments}\label{sec5}
In this section, we present some numerical experiments to confirm our theoretical analysis. 
Assume that the numerical solutions has the long-time decay rate $\|U_{n}\|_{L^{s}(\Omega)}=O\left( t_{n}^{(\alpha+\beta)/\gamma} \right)$ as $t_{n}\to \infty$.
Then we can define the the index function  as
\begin{equation}\label{q_alpha}
    r_{\alpha}^{*}(t_{n})=-\frac{\ln(\|U_{n}\|_{L^{s}(\Omega)}/\|U_{n-m}\|_{L^{s}(\Omega)})}{\ln(t_{n}/t_{n-m})}~~for~m\in \mathbb{N}^{+},
\end{equation}
which is the numerical observation decay rate approximates the theoretical predictions decay rate with order $\frac{\alpha+\beta}{\gamma}$. See more details 
in \cite{CM-preserving}.

\subsection{Nonlinear F-ODEs with time-dependent coefficients}
In this test, consider the scalar model \eqref{fode}: $D_{t}^{\alpha}y(t)+\nu t^{\beta}y(t)^{\gamma}=0$ with initial value $y(0)=y_{0}>0$ and $\beta>-\alpha$, which has an optimal numerical decay rates $O\left( t_{n}^{(\alpha+\beta)/\gamma} \right)$ predicted by Theorem \ref{uni-decay}.

Here we take $m=5$, $u_{0}=5$ and the parameter $\gamma=\{1/4,1/2,1,2,4\}$. Table \ref{tableexam1} and Table \ref{tableexam2} respectively show the polynomial decay rate of the numerical solution computed by Gr\"unwald-Letnikov scheme for various parameters of $\alpha$ and $\beta$.
We also plot the log-log relation of the numerical solution $U_{n}$ and $t_{n}$ in Figure \ref{f1} with various parameters.
It can be seen from those data and Figure that the numerical solutions are monotone and positive and are completely consistent with our theoretical prediction, i.e. $r_{\alpha}^{*}(t_{n})\sim\frac{\alpha+\beta}{\gamma}$.

\begin{table}[H]
\caption{Observed $r^{*}_{\alpha}(t_{n})$ for $t_{n}=2000,4000,6000,8000,10000$ with $\beta=0.3$, $\alpha=0.4$ }
\begin{center}
\begin{tabular}{cccccccc }
\hline $t_n$  & $\gamma=1/4$ & $\gamma=1/2$  & $\gamma=1$ & $\gamma=2$ & $\gamma=4$\\
\hline
$2000$  &2.801315 &1.400861 &0.700384 &0.348185 &0.172802\\
$4000$  &2.800657 &1.400434 &0.700247 &0.348575 &0.173056\\
$6000$  &2.800438 &1.400291 &0.700191 &0.348764 &0.173191\\
$8000$  &2.800328 &1.400219 &0.700158 &0.348883 &0.173282\\
$10000$ &2.800263 &1.400176 &0.700137 &0.348967 &0.173348\\
\hline
$r_{\alpha}^{*}$ &2.8 &1.4 &0.7 &0.35 &0.175\\
\hline
\end{tabular}
\end{center}
\label{tableexam1}
\end{table}

\begin{table}[H]
\caption{Observed $r^{*}_{\alpha}(t_{n})$ for $t_{n}=2000,4000,6000,8000,10000$ with $\beta=-0.3$, $\alpha=0.8$}
\begin{center}
\begin{tabular}{cccccccc }
\hline $t_n$  & $\gamma=1/4$ & $\gamma=1/2$  & $\gamma=1$ & $\gamma=2$ & $\gamma=4$\\
\hline
$2000$  &2.001509 &1.001087 &0.502494 &0.250823 &0.124179\\
$4000$  &2.000754 &1.000559 &0.501724 &0.250673 &0.124242\\
$6000$  &2.000503 &1.000379 &0.501393 &0.250602 &0.124278\\
$8000$  &2.000378 &1.000288 &0.501199 &0.250556 &0.124303\\
$10000$ &2.000302 &1.000232 &0.501067 &0.250524 &0.124322\\
\hline
$r_{\alpha}^{*}$&2 &1 &0.5 &0.25 &0.125\\
\hline
\end{tabular}
\end{center}
\label{tableexam2}
\end{table}

\begin{figure}[!htb]
\begin{centering}
\begin{minipage}[c]{0.5\textwidth}
\centering\includegraphics[width=7.5cm,height=7cm]{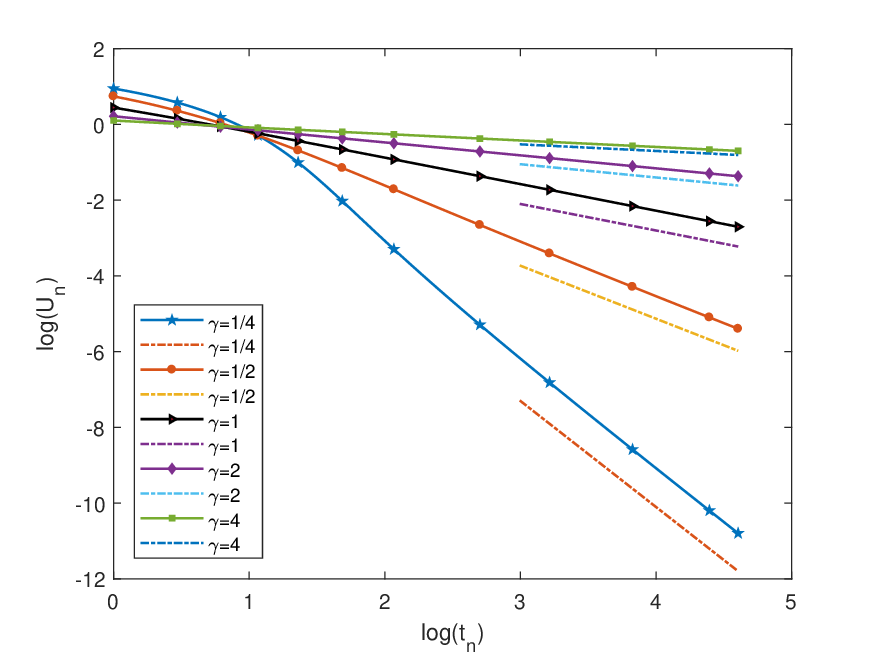}
\renewcommand{\figurename}{Figure.}
\end{minipage}
\begin{minipage}[c]{0.5\textwidth}
\centering\includegraphics[width=7.5cm,height=7cm]{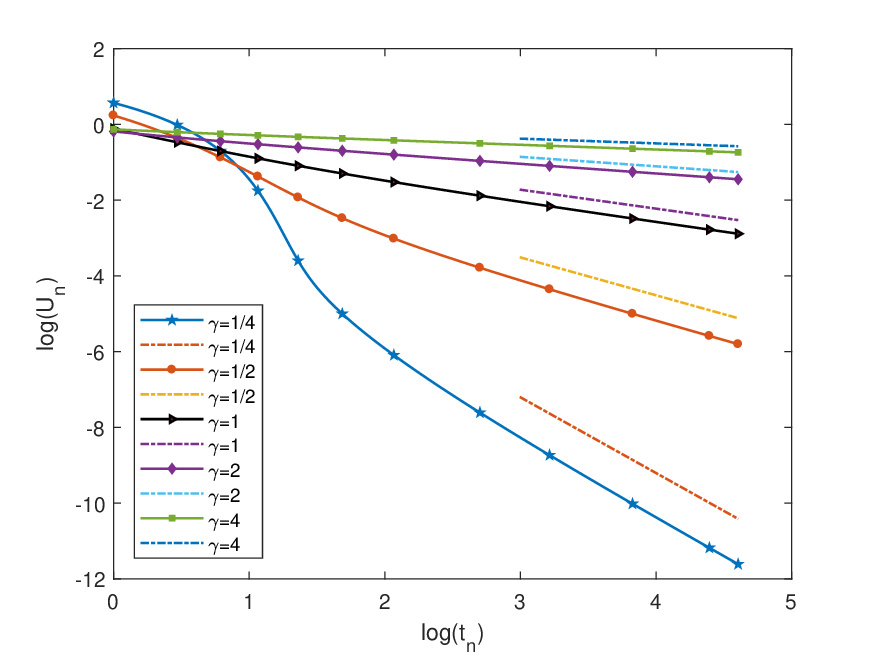}
\renewcommand{\figurename}{Figure.}
\end{minipage}
\end{centering}
\caption{ Logarithm of numerical solutions $\log(U_{n})$ for $\gamma=1/4,1/2,1,2,4$ with $\beta=0.3$, $\alpha=0.4$ (left) and $\beta=-0.3$, $\alpha=0.8$ (right). \label{f1}}
\end{figure}

\subsection{Time fractional PDEs with time-dependent coefficients}
Consider the time fractional partial differential equation:
\begin{equation}\label{pde_1}
    \left\{\begin{aligned}
        &D_{t}^{\alpha}u(t,x)-t^{\beta}\Delta u(t,x)=\frac{t^{\beta}}{(1+t)^{\alpha+\beta}}~~for~t>0~and~x\in \Omega,\\
        &u(t,x)=0~~for~t>0~and~x\in \partial\Omega,\\
        &u(0,x)=u_{0}(x)~~for~x\in \Omega,
    \end{aligned}\right.
\end{equation}
where $\mathcal{N}(u)=\Delta u$ is Laplace operator. One has the Poincar\'e inequality  $\|u(t,\cdot)\|_{L^{2}(\Omega)}\leq \kappa(\int_{\Omega}|\bigtriangledown u(t,\cdot)|^{2}dx)^{1/2}$ for some positive constant $\kappa>0$, i.e., the structural assumption \eqref{str-ass1} holds true for $s=2,~\gamma=1$. Therefore, we can obtain that the decay rate of the true solution is $O(t^{-(\alpha+\beta)})$.

In this example, we take $\Omega=[0,1]^{2}$ and the initial value $u_{0}(x_{1},x_{2})=10\sin (\pi x_{1})\sin (\pi x_{2})$.
The linear finite element in space and Gr\"unwald-Letnikov scheme in time are employed to solve the problem \eqref{pde_1} with time step size $\tau=0.1$. 
Table \ref{tablepde2_1} and Table \ref{tablepde2_2} presents the results of the decay rate for numerical solutions with $\beta=-0.3,~-0.1,~0.1,~0.3,~0.6$, $\alpha=0.5$ and $\beta=0.3$, $\alpha=0.1,~0.3,~0.5,~0.7,~0.9$ to problem \eqref{pde_1}. 
Both log-log picture presented in Figure \ref{2dpde} show the evolution of $V_{n}:=\|U_{n}\|_{L^{2}(\Omega)}$ has decay rate $t_{n}^{-(\alpha+\beta)}$,
which are also completely consistent with our theoretical prediction

\begin{table}[H]
\caption{Observed $r^{*}_{\alpha}(t_{n})$ for $t_{n}=20,60,100,140,180$ with $\beta=-0.3,~-0.1,~0.1,~0.3,~0.6$, and $\alpha=0.5$}
\begin{center}
\begin{tabular}{cccccccc }
\hline $t_n$  & $\beta=-0.3$ & $\beta=-0.1$  & $\beta=0.1$ & $\beta=0.3$ & $\beta=0.6$\\
\hline
$20$  &0.203591~ &0.403624~ &0.603672~ &0.803713~ &1.103757~\\
$60$  &0.201189~ &0.401221~ &0.601252~ &0.801270~ &1.101281~\\
$100$  &0.200701~ &0.400732~ &0.600756~ &0.800768~ &1.100773~\\
$140$  &0.200491~ &0.400521~ &0.600542~ &0.800551~ &1.100554~\\
$180$  &0.200375~ &0.400404~ &0.600423~ &0.800430~ &1.100432~\\
\hline
$r_{\alpha}^{*}$&0.2~ &0.4~ &0.6~ &0.8~ &1.1~\\
\hline
\end{tabular}
\end{center}
\label{tablepde2_1}
\end{table}

\begin{table}[H]
\caption{Observed $r^{*}_{\alpha}(t_{n})$ for $t_{n}=20,60,100,140,180$ with $\beta=0.3$, and $\alpha=0.1,~0.3,~0.5,~0.7,~0.9$}
\begin{center}
\begin{tabular}{cccccccc }
\hline $t_n$  & $\alpha=0.1$ & $\alpha=0.3$  & $\alpha=0.5$ & $\alpha=0.7$ & $\alpha=0.9$\\
\hline
$20$  &0.400381~ &0.601848~ &0.803713~ &1.005938~ &1.208512~\\
$60$  &0.400083~ &0.600622~ &0.801270~ &1.002020~ &1.202882~\\
$100$  &0.400029~ &0.600372~ &0.800768~ &1.001218~ &1.201734~\\
$140$  &0.400008~ &0.600265~ &0.800551~ &1.000872~ &1.201240~\\
$180$ &0.400000~ &0.600205~ &0.800430~ &1.000679~ &1.200965~\\
\hline
$r_{\alpha}^{*}$&0.4~ &0.6~ &0.8~ &1~ &1.2~\\
\hline
\end{tabular}
\end{center}
\label{tablepde2_2}
\end{table}

\begin{figure}[!htb]
\begin{centering}
\begin{minipage}[c]{0.45\textwidth}
\centering\includegraphics[width=8cm,height=7cm]{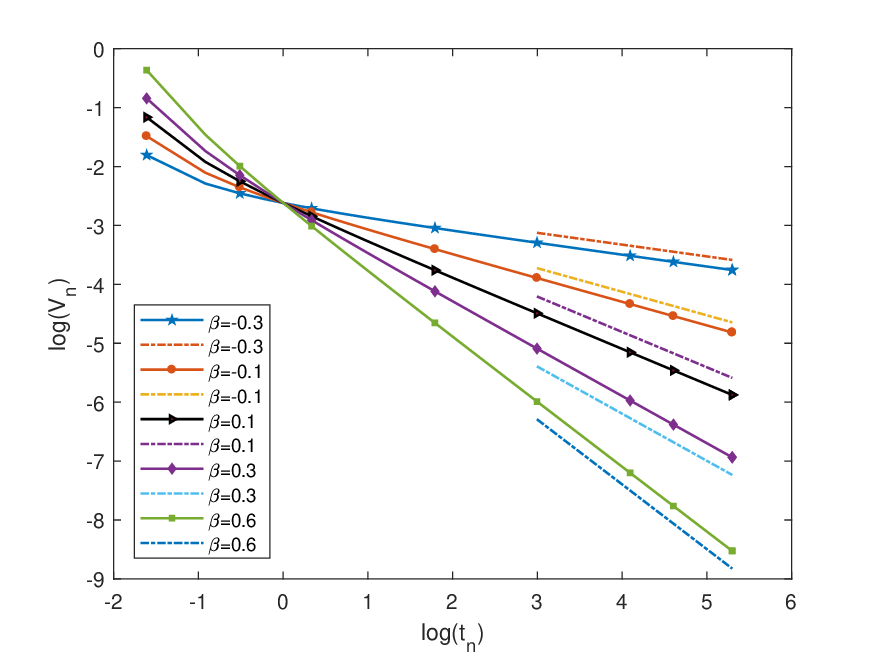}
\renewcommand{\figurename}{Figure.}
\end{minipage}
\begin{minipage}[c]{0.45\textwidth}
\centering\includegraphics[width=8cm,height=7cm]{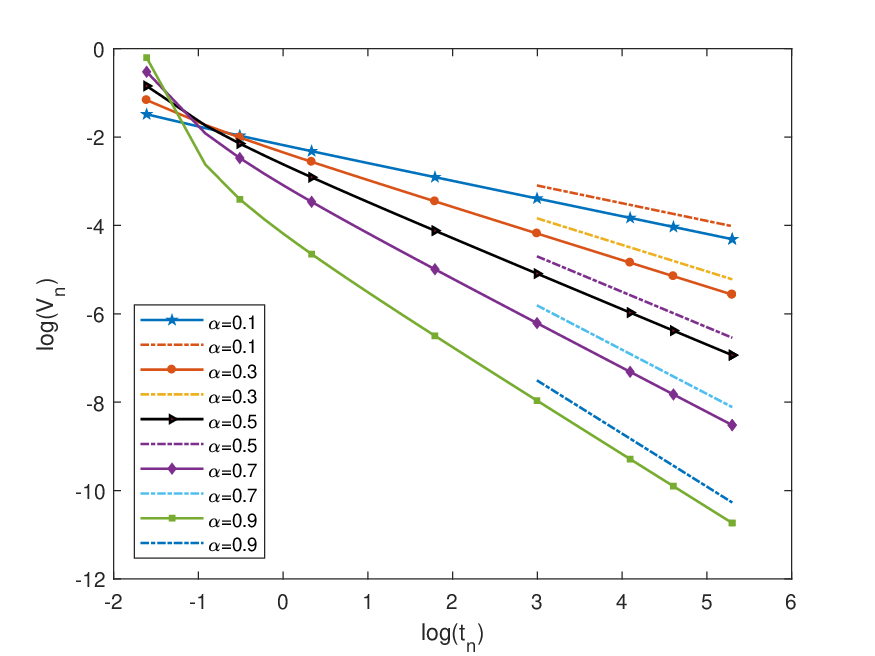}
\renewcommand{\figurename}{Figure.}
\end{minipage}
\end{centering}
\caption{ Logarithm of numerical solutions $\log(V_{n})$ for $\beta=-0.3,-0.1,~0.1,~0.3,~0.6$, $\alpha=0.5$ (left) and $\beta=0.3$, $\alpha=0.1,~0.3,~0.5,~0.7,~0.9$ (right). \label{2dpde}}
\end{figure}

\bibliographystyle{plain}
\bibliography{CM}
\end{document}